\documentclass[11pt,a4paper]{article}

\usepackage{amsmath}
\usepackage{amsthm}
\usepackage{amssymb}
\usepackage{amsfonts}
\usepackage{mathtools}
\usepackage{mathrsfs}

\usepackage{parskip} 
\usepackage{graphicx}  
\usepackage{dcolumn}  
\usepackage{color}  
\usepackage{bm}  
\usepackage[caption=false]{subfig} 
\usepackage{listings} 
\usepackage[c]{esvect} 
\usepackage[retainorgcmds]{IEEEtrantools}
\usepackage{verbatim} 
\usepackage[section]{placeins}
\usepackage{tikz}
\usetikzlibrary{positioning}
\usetikzlibrary{shapes.misc}
\usetikzlibrary{shapes.arrows}
\usetikzlibrary{patterns}
\usetikzlibrary{braids}

\usepackage[]{color}

\definecolor{green}{rgb}{0,1,0.28}

\definecolor{brown}{rgb}{.52,.23,.0}

\usepackage{url}
\usepackage{hyperref}

\newcommand{\Cf}{\mathrm{Conf}}

\newcommand{\qCf}{\overline{\mathrm{Conf}}}

\newcommand{\RR}{\mathbb{R}}
\newcommand{\NN}{\mathbb{N}}

\newcommand{\DD}{\mathcal{D}}
\renewcommand{\SS}{\mathcal{S}}
\newcommand{\ang}{\mathrm{ang}}
\newcommand{\dist}{\mathrm{dist}}
\newcommand{\perm}{\varsigma}

\theoremstyle{definition} \newtheorem{defn}{Definition}[subsection]
\theoremstyle{plain} \newtheorem{thm}[defn]{Theorem}
\theoremstyle{plain} 
\theoremstyle{definition} 
\theoremstyle{plain} \newtheorem{thmm}{Theorem}
\theoremstyle{plain} \newtheorem{prop}[defn]{Proposition}
\theoremstyle{plain} \newtheorem{fact}[defn]{Fact}
\theoremstyle{plain} \newtheorem{lem}[defn]{Lemma}
\theoremstyle{plain} 
\theoremstyle{definition} 
\theoremstyle{remark} \newtheorem{rem}[defn]{Remark}

\begin{document}

\title{Thick braids and other non-trivial homotopy in configuration spaces of hard discs}
\author{Patrick Ramsey \\ School of Mathematical Sciences \\ Lancaster University}
\date{\today}

\maketitle

\begin{abstract}
We study ordered configuration spaces of n hard discs inside a unit disc, and how the topology changes with the radius r of the hard discs. We describe the full homotopy type of this space for all radii when n = 4 and exhibit nontrivial classes in $\pi_{n-3}$ for all n. We also explore the persistence of these nontrivial classes when the ambient disc is deformed into an ellipse.
\end{abstract}

\tableofcontents

\section{Introduction} \label{chap: intro}

We consider $\Cf_{n,r}(U)$, the ordered configuration space of $n$ non-intersecting open discs of radius $r$ in a (typically convex) subset $U\subset \RR^2$. This is a natural generalisation of $\Cf_n(U)$, the ordered configuration space of $n$ points in $U$. However, whereas the topological properties of $\Cf_n(U)$ are well-understood \cite{cohom, aspher, hans} when $U \simeq \RR^2$ is a homotopy-equivalence (indeed, then $\Cf_n(U) \simeq \Cf_n(\RR^2)$), much less is known about $\Cf_{n,r}(U)$, which is topologically highly dependent both on $r$ and $U$. $\Cf_{n,r}(U)$ has been proposed as a model for understanding phase transitions in physics \cite{top2, top1}, where the centre of the disc represents an atom, and the radius is proportional to the strength of the forces between atoms and inversely proportional to the energy of the atom. For this application, there is a particular focus on the \emph{thermodynamic limit} ($n\to \infty$ with the total disc area $n \pi r^2$ fixed). Some work in the field of topological robotics \cite{farb} investigates motion planning algorithms on $\Cf_{n,r}(U)$.

Here, we use the notation $\DD = (D_1, D_2, \ldots, D_n) \in \Cf_{n,r}(U)$ interchangeably with the underlying configuration $x = (x_1, \ldots, x_n) \in \Cf_n(U)$ (where $D_i = B(x_i,r)$), and we write $\bigcup x \coloneqq \bigcup_{i=1}^n \{x_i\}$, $\bigcup \DD \coloneqq \bigcup_{i=1}^n D_i$. We consider primarily the case where $U$ is the unit disc $D^2$, in which case we drop $U$ from the notation. When studying the homotopy groups in \S\ref{chap: crits} and \S\ref{chap: 4}, we concatenate paths right-to-left.

\subsection{Min-type Morse theory and stress graphs}  \label{sec: crits}

A major mathematical step in the understanding of configuration spaces was the application of `min-type' Morse theory \cite{bald}. Baryshnikov \emph{et al.} developed the tautological function $\tau \colon \Cf_n(U) \to \RR$, given by \[ \tau(x) = \min\left\{ \min\left\{\frac{1}{2}|x_i-x_j| \mid i\neq j\right\}, \min\left\{\mathrm{dist} (x_i, \partial U) \mid 1\le i \le n\right\} \right\}\ , \] to play the role of the Morse function. By construction, $\tau(x)$ is the greatest $r \in \RR$ such that if we replace each $x_i$ with the ball $B(x_i,r)$, the resulting configuration $(B(x_1,r), \ldots, B(x_n,r))$ is a valid configuration of $\Cf_{n,r}(U)$. Thus $\Cf_{n,r}(U) = \tau^{-1}[r, \infty)$.

Analagously to Morse theory, we can find critical points of $\tau$, and show that $\Cf_{n,r}(U) \simeq \Cf_{n,s}(U)$ for any $r<s$ such that $(r,s]$ contains no critical points. We show this by finding a vector field $V$ on $\Cf_n(U)$ such that $\tau$ is strictly increasing along the flowlines of $V$ at every non-critical point; then for any $r<r'<s$, we can reparametrise the flow such that it retracts $\Cf_{n,r'}(U)$ onto $\Cf_{n,s}(U)$. The critical points are configurations $x\in \Cf_n(U)$ such that $\tau$ is not increasing to first order along any $v\in T_x\Cf_n(U)$. In particular, consider the following set of continuous functions $\Cf_n(U) \to \RR$,  \[ N \coloneqq \left\{ x\mapsto \frac{1}{2}|x_i-x_j| \mid i\neq j\right\} \cup \left\{ x\mapsto \mathrm{dist} (x_i, \partial U) \mid 1\le i \le n\right \} . \] Then $\tau(x) = \min \{f(x) \mid f\in N\}$. Writing $N_x = \{ f\in N \mid f(x) = \tau(x) \}$, we see that $x\in \Cf_n(U)$ is critical when there is no local direction along which all the functions in $N_x$ increase simultaneously -- that is, the increasing directions of the functions in $N_x$ can be thought of as being in opposition with each other. This motivates the idea of a balanced stress graph. 

\begin{defn} \label{defn: stress}
    Let $x \in \Cf_{n}(U)$. The \emph{stress graph} of $x$, $N(x)$, consists of:
    \begin{itemize}
        \item A vertex at each $x_i$, called an \emph{interior vertex},
        \item For each function of the form $x\mapsto \frac{1}{2} \left|x_i-x_j \right|$ in $N_x$, a straight \emph{interior edge} joining the vertices at $x_i$ and $x_j$,
        \item For each  function of the form $x\mapsto \dist(x_i, \partial U)$ in $N_x$, a \emph{boundary vertex} at each point on $\partial U$ where this distance is realised, and a straight  \emph{boundary edge} joining each of these boundary vertices to the vertex at $x_i$ (each boundary edge will be orthogonal to $\partial U$).
    \end{itemize}
\end{defn}

\begin{figure}[htb]
\begin{center}
    \begin{tikzpicture}[node distance={15mm}, thick, main/.style={}]
            \draw[color=black] (0,0) circle [radius=2];
            \draw[color=black!50, densely dashed] (-0.4,-0.3) circle [radius=0.5] node[circle, inner sep = 0, minimum size = 6, fill=black] {};
            \draw[color=black!50, densely dashed] (-0.5,1) circle [radius=0.5];
            \draw[color=black!50, densely dashed] (0.37,0.5) circle [radius=0.5];
            \draw[color=black!50, densely dashed] (0,-1.5) circle [radius=0.5];
            \draw[color=black!50] (-0.5,1) node[circle, inner sep = 0, minimum size = 6, fill=black] {} -- (0.37,0.5) node[circle, inner sep = 0, minimum size = 6, fill=black] {};
            \draw[color=black!50] (0,-1.5) node[circle, inner sep = 0, minimum size = 6, fill=black] {} -- (0,-2) node[circle, inner sep = 0, minimum size = 6, fill=black!50] {};
    \end{tikzpicture}
\end{center}
\caption{A configuration $x$ of points in $D^2$ (black) with its stress graph (grey). Each point of $x$ is contained in a circle (grey, dashed), where the radius of these discs is the greatest achievable without any discs overlapping each other or leaving the unit disc. At each point of contact between discs, or between a disc and $\partial D^2$, an edge joins the corresponding vertices.} \label{fig: stress}
\end{figure}

\begin{defn} \label{defn: bal}
    Take $x\in \Cf_n$, and let $E$ be the set of edges of $N(x)$. A \emph{weighting} on $N(x)$ is a map $w\colon E \to [0, \infty)$, where $w_e \coloneqq w(e)$ defines the force with which each edge pushes on its endpoints. More precisely, given an edge $e$ incident to $x_i$, let $v_{i,e}$ denote the other vertex on $e$ --  so $v_{i,e}$ is either $x_j$ for some $j\neq i$, or $\frac{x_i}{|x_i|}$ if $e$ is a boundary edge. Then $e$ exerts a force $w_e(x_i - v_{i,e})$ on $x_i$. If $x_i$ is adjacent to the boundary via an edge $e$, then $e$ exerts a force $w_e\left( \frac{x_i}{|x_i|} - x_i\right)$ on the boundary.
    
    We say that $w$ is \emph{balanced} if the total force $F_i \coloneqq \sum_{e \textrm{ incident to } x_i} w_e(x_i - v_{i,e})$ on each interior vertex $x_i$ is zero. $N(x)$ is a \emph{balanced stress graph} if it has a non-trivial balanced weighting.
\end{defn}

In \cite{bald}, the authors give a second condition in the definition of a balanced weighting: in each connected component $G \subset N(x)$, the total force exerted on the boundary by the boundary edges is zero. However, this condition is a consequence of the condition that $F_i=0$ for all $i$. 

We will typically use the term critical configuration to refer to all configurations with a balanced stress graph, and we will sometimes treat $w$ as part of the information contained in the stress graph. In summary, we obtain:

\begin{fact} \label{fact: crits}
    \cite[Thm. 4.2]{bald} If $[r,s)$ does not contain a critical radius of $\tau$, then $\Cf_{n,r}(U) \simeq \Cf_{n,s}(U)$. The critical radii are characterised by balanced stress graphs.
\end{fact}

Around the same time as this Morse-theoretic approach was being developed, authors began trying to find these critical configurations for small $n$, e.g. \cite{comptop}, using numerical techniques. These approximate the repulsive forces from the stress graph by a smooth function inversely proportional to $r^n$, where $r$ is the distance between two given points, using this to define a flow on the configuration space. The greater the value of $n$, the more like the hard disc case the simulation behaves. The authors aimed to find `wells of attraction' - configurations where the total force on each particle is (nearly) zero, and where nearby configurations naturally flow towards these wells. These can then be confirmed as critical points analytically. However it is not certain that this technique finds all critical points.

The most significant research into the case $U=D^2$ comes from Alpert \cite{alp}, who considers the natural inclusion $\iota \colon \Cf_{n,r} \to \Cf_n$ and its pullback to cohomology $\iota^* \colon H^*(\Cf_{n}) \to H^*(\Cf_{n,r})$. Using geometric arguments, Alpert proves:
\begin{fact}
    The least critical radius of $\Cf_n$ is $\frac{1}{n}$. The stress graph of the corresponding configuration is a diameter of $D^2$ containing all $n$ points -- that is, all the points lie on a common diameter, and there is some non-repeating sequence $i_1, \ldots, i_n$ containing the integers in $[1,n]$ such that $x_{i_1}$ and $x_{i_n}$ are at distance $\frac{1}{n}$ from the boundary, and $|x_{i_k}-x_{i_{k+1}}| = \frac{2}{n}$ for all $1\le k \le n-1$.
\end{fact}
\begin{fact} \label{fact: critnd}
    If $\Cf_n$ has a critical configuration whose stress graph is not just a diameter, then the corresponding radius is greater than $\frac{3}{2n+3}$.
\end{fact}
Alpert then considers the change in $\ker \iota^*$ as $r$ increases. The author shows that $\ker \iota^*$ becomes non-trivial at the first critical radius, $\frac{1}{n}$, and in the case $n=4$, it also increases in size at $r=\frac{1}{3}$ and $r = \sqrt{2}-1$. We show in \S\ref{chap: crits} that this is all the critical radii of $\Cf_4$.

The majority of recent research in the area has focused broadly on finding additional structure in the homology and cohomology of the configuration spaces of hard discs inside an infinite strip \cite{disk1, disk2, disk3, disk4, disk5}, although some work has also looked at topological complexity \cite{disk6}. There also exists some similar research into configuration space of polygons, primarily squares \cite{sq1, sq2, sq3, sq4}.

\subsection{Results}

In \S\ref{chap: crits}, we calculate the critical configurations and radii of $\tau \colon \Cf_n \to \RR$ for small $n$. 
\begin{thmm}
    The full set of pairs $(n,r)$, where $n\in \{1,2,3,4,5\}$ and $r$ is a critical radius of $\Cf_n$, is as shown in Fig. \ref{fig: gcrits}.
\end{thmm}
This is seen in Prop. \ref{prop: crit1}, \ref{prop: crit2}, \ref{prop: crit3}, \ref{prop: crit4}, and \ref{prop: crit5}.

In \S\ref{chap: nts}, we look just beyond the first critical radius for topological features which vanish in $\Cf_n$ when pushed forward through the inclusion $\iota \colon \Cf_{n,r} \to \Cf_n$. Thm. \ref{thm: braid} and \ref{thm: nts} show
\begin{thmm}  \label{thmm: nts}
    Let $n\ge 4$, and $r = \frac{1}{n}+ \varepsilon$ for sufficiently small $\varepsilon>0$. $\pi_{n-3}(\Cf_{n,r})$ contains a non-trivial element which lies in $\ker(\iota_* \colon \pi_{n-3}(\Cf_{n,r}) \to \pi_{n-3}(\Cf_n))$.
\end{thmm}
In particular, for $n=4$ we demonstrate a non-trivial `thick braid' which vanishes as a braid in the (pure) braid group. 

In \S\ref{chap: 4}, we calculate the full homotopy type of $\Cf_{4,r}$ just beyond the first critical radius, culminating in Thm. \ref{thm: graph2}, which says:
\begin{thmm}
    There is a 1-dimensional cell complex $Y$ of Euler characteristic 11 such that $\Cf_{4, r} \simeq Y\times S^1$ for all $r\in \left( \frac{1}{4}, \frac{1}{3} \right]$.
\end{thmm}

Finally, in \S\ref{chap: deform}, we explore deformations of the unit disc. In particular, we consider deformation into an ellipse, which we classify by its eccentricity $e=\sqrt{1-\frac{b^2}{a^2}}$, where $a$ and $b$ are the semi-major and semi-minor radii respectively.  In Thm. \ref{thm: ell} and \ref{thm: ell2}, we show
\begin{thmm}
    Let $n\ge 4$, and $E$ be the ellipse with semi-major radius $a$, semi-minor radius $b$ and eccentricity $e$. Then $\Cf_{n,r}(E)$ contains a non-contractible $(n-3)$-sphere of the same construction as that in Thm. \ref{thmm: nts} for $r=\frac{1}{n}+\varepsilon$, $\varepsilon>0$ sufficiently small, when either: 
    \begin{itemize}
        \item[i.] $b\in \left(\frac{1}{\sqrt{n}}, 1\right]$ and $a=1$, or
        \item[ii.] $b\in \left(\frac{1}{n}, \sqrt{r}\right]$ and $e^2 = \frac{(1-r)^2}{b^2-r^2+(1-r)^2}$.
    \end{itemize}
\end{thmm}
In case \emph{ii}, we show that the homotopy class remains for $e$ arbitrarily close to $1$, or equivalently for $a$ arbitrarily large.

\subsection{Acknowledgements}

I would like to thank my PhD supervisor, Dr. Jonny Evans, for his support and advice. I am grateful to Dr. Hannah Alpert for giving me feedback on an early draft. This research was funded by EPSRC.

\section{The critical radii for small numbers of discs} \label{chap: crits}

In this section, we find the balanced configurations for $\Cf_n, \ n \in \{2,3,4,5\}$ and the associated critical radii, as summarised in Fig. \ref{fig: gcrits}. Let $O(x) \subset N(x)$ be the subgraph induced by the interior vertices. For any weighted graph $(G,w)$, let $Z(G) \subset G$ be the subgraph containing all the vertices of $G$, and all the edges of $G$ with non-zero weight. We start with the following claim.

\begin{lem} \label{lem: hull} \cite[Lemma 4.1]{bald}
	Let $x\in \Cf_n$ be a critical configuration, and choose a non-trivial balanced weighting on $N(x)$. Then:
	\begin{itemize}
		\item[i.] Each non-trivial connected component of $Z(N(x))$ is contained in the interior of the convex hull of its boundary vertices, and
		\item[ii.] Each non-isolated interior vertex in $Z(N(x))$ is contained in the interior of the convex hull of its adjacent vertices.
	\end{itemize}
\end{lem}

\begin{prop} \label{prop: crit1}
	The unique critical radius of $\Cf_1$ is $r=1$. 
\end{prop}
\begin{proof}
	If $x_1 \neq 0$, then there is a unique closest point to $x_1$ on $\partial D^2$, and the vertex at $x_1$ is a leaf. Thus the only balanced choice of weights is trivial by Lemma \ref{lem: hull}\emph{ii}. Conversely, if $x_1=0$, then there is an edge of length 1 joining $x_1$ to each $p\in \partial D^2$, and these are balanced by giving every edge weight 1.
\end{proof}

\begin{prop} \label{prop: crit2}
	The unique critical radius of $\Cf_2$ is $r=\frac{1}{2}$.
\end{prop}
\begin{proof}
	Let $x\in \Cf_2$ and choose a weighting $w$ on $N(x)$. If $Z(N(x))$ does not contain an edge joining $x_1$ and $x_2$, then at least one vertex, say $x_1$, is adjacent to a boundary vertex. Then applying Lemma \ref{lem: hull}\emph{ii} to $x_1$, we see that $w$ is not balanced.
	
	Now suppose $w$ is balanced, so that $Z(N(x))$ contains an edge between $x_1$ and $x_2$. By Lemma \ref{lem: hull}\emph{i}, $x_1$ and $x_2$ must both be adjacent to the boundary in $Z(N(x))$, and the three edges must form a chord of $D^2$. Since boundary edges are orthogonal to the boundary, $N(x)$ is a diameter, and the associated critical radius is $\frac{1}{2}$. We can balance this by giving the interior edge weight 1 and the boundary edges weight 2. 
\end{proof}

We denote by $P_n$ the path on $n$ vertices (and $n-1$ edges), and by $C_n$ ($n\ge 3$) the cyclic graph on $n$ vertices. 

\begin{lem} \label{lem: 2}
	Take $n\ge3$, and let $x\in \Cf_n$ be balanced. Then each non-trivial connected component of $Z(O(x))$ contains a $P_3$ as a subgraph. 
\end{lem}
\begin{proof}
	Take $n\ge 2$, and let $x\in \Cf_n$ be balanced. Choose a non-trivial balanced weighting on $N(x)$, and suppose that $Z(N(x))$ has a connected component $G$ with at most one interior edge. Assume the vertices here are $x_1$ and $x_2$.  By the arguments in Prop. \ref{prop: crit2}, $G$ is a diameter of the unit disc, and the associated critical radius is $\frac{1}{2}$. However, there is no point $p$ in $D^2$ satisfying $|p-x_1| \ge 1$, $|p-x_2|\ge1$ and $\dist(p, \partial D^2) \ge \frac{1}{2}$ simultaneously. Thus $n=2$.
\end{proof}

\begin{prop} \label{prop: crit3}
    The critical radii of $\Cf_3$ are $\frac{1}{3}, 2\sqrt{3}-3$.
\end{prop}
\begin{proof}
    Let $x\in \Cf_3$ be a critical configuration and choose a balanced weighting on $N(x)$. Then $Z(O(x)) = P_3$ or $Z(O(x)) = C_3$ by Lemma \ref{lem: 2}. 

    If $Z(O(x)) = C_3$, then this is an equilateral triangle, so Lemma \ref{lem: hull}\emph{ii} implies that each vertex must be adjacent to a boundary vertex. Since the boundary edges all have the same length, the triangle must be centred at the origin. We can balance this configuration by giving the interior edges weight 1 and the boundary edges weight $2\sqrt{3}$. This configuration corresponds to the critical radius $2\sqrt{3}-3$.

    If $Z(O(x)) = P_3$, then applying Lemma \ref{lem: hull}\emph{ii} at each vertex in succession implies that the two vertices of degree 1 must be adjacent to a boundary vertex, and moreover, the four edges must all be collinear. Then $N(x)$ is a diameter of $D^2$. We can balance this by giving the interior edges weight 1 and the boundary edges weight 2, and the associated radius is $\frac{1}{3}$.
\end{proof}

\begin{prop} \label{prop: crit4}
    The critical radii of $\Cf_4$ are $\frac{1}{4}, \frac{1}{3}, \sqrt{2}-1$.
\end{prop}
\begin{proof}
    Let $x\in \Cf_4$ be a critical configuration and choose a balanced weighting $w$ on $N(x)$. By Lemma \ref{lem: 2}, each component of $Z(O(x))$ contains a $P_3$. We consider the following cases.

    Suppose $Z(O(x)) \supset C_4$. Then this appears as a rhombus. By Lemma \ref{lem: hull}\emph{ii}, each vertex must be adjacent to a boundary vertex. Then $N(x)$ must be a square centred at the origin. This can be balanced by giving the interior edges weight $1$ and the boundary edges weight $2\sqrt{2}$, and corresponds to the critical radius $\sqrt{2}-1$. 

    Next suppose $Z(O(x))$ contains a $C_3$ (as an equilateral triangle) but no $C_4$. Assume without loss of generality that the $C_3$ contains $x_1, x_2, x_3$. By Lemma \ref{lem: hull}\emph{ii}, each vertex of the $C_3$ must have a third incident edge in $Z(N(x))$. These cannot all be boundary edges, since these three vertices would form the critical configuration of radius $2\sqrt{3}-3$ from $\Cf_3$, and then there is insufficient space to fit the fourth point. Therefore $x_4$ is adjacent to exactly one vertex, say $x_3$, by an edge of non-zero weight. $x_4$ is also adjacent to a boundary vertex by Lemma \ref{lem: hull}\emph{ii}, and the two edges incident to $x_4$ are collinear. Then there are two possibilities: \newline
    1) $x_3=0$. Then the edge connecting $x_1$ to $x_3$ is collinear with the boundary edge incident to $x_1$. Thus $w$ is not balanced by Lemma \ref{lem: hull}\emph{ii} applied to $x_1$. \newline
    2) $x_3 \neq 0$. Since $x_1$ and $x_2$ are both at distance $r$ from the boundary, it follows that the line containing $x_3$ and $x_4$ is a line of symmetry for the configuration, whence we derive that $r$ satisfies $r^2 + \left( \left(3+\sqrt{3} \right)r -1\right)^2 = (1-r)^2$. This has solutions $r\in \{0, \frac{1}{3}\}$, but $0 \notin \tau(\Cf_4)$, so $r=\frac{1}{3}$. Thus $x_3 = 0$, a contradiction.

    Finally, suppose $Z(O(x))$ contains no $C_3$ or $C_4$. If it is $P_4$, then then applying Lemma \ref{lem: hull}\emph{ii} at each vertex in succession implies that $N(x)$ is a diameter containing all four points. This is balanced by giving the interior edges weight 1 and the boundary edges weight 2, and has associated radius $\frac{1}{4}$. If the interior of $Z(N(x))$ consists of a $P_3$ and an isolated point, then $Z(N(x))$ is a diameter containing three points by the same arguments, corresponding to radius $\frac{1}{3}$. Otherwise, the interior of $Z(N(x))$ consists of one vertex of degree 3, and three leaves. Then each leaf is incident to a boundary edge, which is collinear with its other incident edge, by Lemma \ref{lem: hull}\emph{ii}. This means that the vertex of degree 3 lies on three radial lines of equal length and is thus at the origin, so this corresponds to radius $\frac{1}{3}$ again.
\end{proof}

\begin{prop} \label{prop: crit5}
    The critical radii of $\Cf_5$ are $\frac{1}{5}, \frac{1}{4}, \frac{1}{3}, \frac{\sin \frac{\pi}{5}}{1 + \sin \frac{\pi}{5}}$.
\end{prop}
\begin{proof}
    Let $x\in \Cf_5$ be a critical configuration and choose a balanced weighting $w$ on $N(x)$. By Lemma \ref{lem: 2}, each component of $Z(O(x))$ contains a $P_3$. We consider the following cases.

    First, suppose $Z(O(x))$ contains a $C_5$. This is a pentagon with 5 equal edges. By Lemma \ref{lem: hull}\emph{ii}, each vertex must have a third incident edge in $Z(N(x))$, and these can be boundary edges or diagonals within the pentagon. \newline
    1) If there are no diagonals, then the $C_5$ is a regular pentagon centred at the origin. This is balanced by giving each interior edge weight 1, and each boundary edge weight $4\cos\frac{3\pi}{10}$, and corresponds to a critical radius $r = \frac{\sin \frac{\pi}{5}}{1 + \sin \frac{\pi}{5}}$. \newline
    2) If there is a diagonal, then one of the vertices on the diagonal does not satisfy the condition of Lemma \ref{lem: hull}\emph{ii}, so $w$ is not balanced.
    
    Next, suppose $Z(O(x))$ contains a $C_4$, but no $C_5$. Assume the vertices of the $C_4$ are $x_1, x_2, x_3, x_4$ in anticlockwise order. If the fifth point is not in the same component of $Z(N(x))$, then the $C_4$ would have to be the square with associated radius $\sqrt{2}-1$ from the case $n=4$ -- however, this is impossible as $x_5$ will not fit at a sufficient distance from the $C_4$ and the boundary. Therefore, $x_5$ is adjacent to exactly one vertex in the cycle, say $x_4$. Then $x_5$ must have a boundary edge collinear with the edge between $x_4$ and $x_5$ by Lemma \ref{lem: hull}\emph{ii}, and each of the remaining three vertices must be adjacent to the boundary. This yields three subcases: \newline
    1) $x_1$ and $x_3$ are each adjacent to a boundary vertex and $x_4=0$. If such a balanced weighting exists, it has radius $r=\frac{1}{3}$. We prove this is a critical radius in the final paragraph. \newline
    2) $x_1$ and $x_3$ are each adjacent to a boundary vertex and $x_4\neq0$. Since $\dist(x_1, \partial D^2) = \dist(x_3, \partial D^2)$ and $|x_1-x_4| = |x_3-x_4|$, it follows that $x_1$ and $x_3$ are equidistant from the diameter containing $x_4$ and $x_5$, and so $x_2$ lies on this diameter. Then by summing distances along this diameter, we have \[ 2 = \dist(x_5, \partial D^2)  + |x_4-x_5| + |x_2-x_4| + \dist(x_2, \partial D^2) \ge 6r \] However, since the edge from $x_4$ to $x_3$ is not radial and therefore not collinear with the boundary edge incident to $x_3$, we have \[ \dist(x_4, \partial D^2) < |x_3-x_4| + \dist(x_3, \partial D^2) = 3r \] $x_4$ lies on a diameter at a distance $3r$ from one end, so $\dist(x_4, \partial D^2) = \min\{3r, 2-3r\}$. Therefore $2-3r<3r$, which is a contradiction -- this case does not correspond to the stress graph of any critical $x\in \Cf_5$.

    Next, suppose $Z(O(x))$ contains a $C_3$, but no $C_4$ or $C_5$. Denote the vertices of the $C_3$ as $x_1, x_2, x_3$ in anticlockwise order, and the remaining vertices $x_4, x_5$. If the $C_3$ is an entire component of $Z(O(x))$, then we obtain the critical configuration of radius $2\sqrt{3}-3$ from $\Cf_3$, but this is too big to fit the remaining points. If the component containing $C_3$ contains 4 vertices, then the fourth must be adjacent to exactly one vertex of the $C_3$, but we argue in the proof for $\Cf_4$ that this configuration is never critical. Therefore, this component contains all 5 vertices. The graphs of this type can be considered in three cases. \newline
    1) The $C_3$ contains a vertex, say $x_3$, which is adjacent to $x_4$ and $x_5$. Then both are adjacent to a boundary vertex, and the two incident edges at each vertex are collinear, by Lemma \ref{lem: hull}\emph{ii}. Therefore $x_3$ lies either on two non-collinear radial lines, or at the centre of a diameter, so $x_3=0$. If such a balanced weighting exists, it has radius $r=\frac{1}{3}$. We prove this is a critical radius in the final paragraph. \newline
    2) One of the vertices of $C_3$, say $x_3$, is adjacent to one vertex, say $x_4$, and $x_4$ is adjacent to $x_5$. Then $x_5$ is adjacent to a boundary vertex, and the section of graph between $x_3$ and this boundary vertex is a straight line by Lemma \ref{lem: hull}\emph{ii}; moreover, $x_1$ and $x_2$ are adjacent to boundary vertices. This latter is impossible if $x_3=0$; therefore, $x_1$ and $x_2$ are equidistant from the line passing through $x_3$ and $x_5$. But $x_1$ does not lie in the convex hull of its adjacent boundary vertex, $x_2$ and $x_3$, so there is no such balanced weighting by Lemma \ref{lem: hull}\emph{ii}.  \newline
    3) Two vertices of the triangle, say $x_2$ and $x_3$, are each adjacent to one of the remaining vertices, say $x_4, x_5$ respectively. Then $x_4, x_5$ are each adjacent to a boundary vertex, and the edge joining them to $x_2$ and $x_3$ respectively is collinear with their boundary edge by Lemma \ref{lem: hull}\emph{ii}. Then $\dist(x_2, \partial D^2) = \dist(x_3, \partial D^2) = 3r$. However, we also have $\dist(x_2, \partial D^2) \le |x_2-x_1| + \dist(x_1, \partial D^2) = 3r$ and $\dist(x_3, \partial D^2) \le |x_3-x_1| + \dist(x_1, \partial D^2) = 3r$. At least one of these is not a straight line distance, and so the inequality is strict, which is a contradiction. Therefore this case does not correspond to the stress graph of a critical $x\in \Cf_5$.

    Finally, suppose $Z(O(x))$ contains no cycles. Using arguments from the proof for $\Cf_4$, we can say that if $Z(O(x))$ is a $P_5$, then $Z(N(x))$ is a diameter of $D^2$, corresponding to $r=\frac{1}{5}$; if $Z(O(x))$ is the union of a $P_4$ with an isolated vertex, then the non-trivial component of $Z(N(x))$ is a diameter of $D^2$, corresponding to $r=\frac{1}{4}$; and if the longest path in $Z(O(x))$ is a $P_3$, then there is a vertex at the origin and the corresponding radius is $r = \frac{1}{3}$. The only remaining candidate for $Z(O(x))$ is the graph containing a $P_4$ with the fifth vertex adjacent to one of the middle vertices. By Lemma \ref{lem: hull}\emph{ii}, each of the leaves of this graph has an incident boundary edge. We see that the vertex of degree 3 lies on three radial lines, not all of the same length, which is a contradiction.
\end{proof}

\begin{figure}
    \centering
    \begin{tikzpicture}
        \begin{scope}[scale=0.6];
        \begin{scope}[shift={(-3,7)}];
            \draw[color=black] (0,0) circle [radius=2];
            \draw[fill=blue!10] (0,0) circle [radius=2] node[circle, inner sep = 0, minimum size = 3, fill=black] {};;
            \node at (0,-3) {$1$};
        \end{scope};
        \begin{scope}[shift={(2,7)}];
            \draw[color=black] (0,0) circle [radius=2];
            \draw[fill=blue!10] (-1,0) circle [radius=1] node[circle, inner sep = 0, minimum size = 3, fill=black] (2) {};
            \draw[fill=blue!10] (1, 0) circle [radius=1] node[circle, inner sep = 0, minimum size = 3, fill=black] (3) {};
            \node[circle, inner sep = 0, minimum size = 3, fill=black] at (-2,0) (1) {};
            \node[circle, inner sep = 0, minimum size = 3, fill=black] at (2,0) (4) {};
            \draw[color=black!50] (1) -- (2) -- (3) -- (4);
            \node at (0,-3) {$\frac{1}{2}$};
        \end{scope};
        \begin{scope}[shift={(7,7)}];
            \draw[color=black] (0,0) circle [radius=2];
            \draw[fill=blue!10] (-1.33,0) circle [radius=0.67] node[circle, inner sep = 0, minimum size = 3, fill=black] (2) {};
            \draw[fill=blue!10] (0, 0) circle [radius=0.67] node[circle, inner sep = 0, minimum size = 3, fill=black] (3) {};
            \draw[fill=blue!10] (1.33, 0) circle [radius=0.67] node[circle, inner sep = 0, minimum size = 3, fill=black] (4) {};
            \node[circle, inner sep = 0, minimum size = 3, fill=black] at (-2,0) (1) {};
            \node[circle, inner sep = 0, minimum size = 3, fill=black] at (2,0) (5) {};
            \draw[color=black!50] (1) -- (2) -- (3) -- (4) -- (5);
            \node at (0,-3) {$\frac{1}{3}$};
        \end{scope};
        \begin{scope}[shift={(12,7)}];
            \draw[color=black] (0,0) circle [radius=2];
            \draw[fill=blue!10] (0, 1.07) circle [radius=0.93] node[circle, inner sep = 0, minimum size = 3, fill=black] (2) {};
            \draw[fill=blue!10] (0.93, -0.54) circle [radius=0.93] node[circle, inner sep = 0, minimum size = 3, fill=black] (3) {};
            \draw[fill=blue!10] (-0.93, -0.54) circle [radius=0.93] node[circle, inner sep = 0, minimum size = 3, fill=black] (4) {};
            \node[circle, inner sep = 0, minimum size = 3, fill=black] at (0,2) (1) {};
            \node[circle, inner sep = 0, minimum size = 3, fill=black] at (1.73,-1) (5) {};
            \node[circle, inner sep = 0, minimum size = 3, fill=black] at (-1.73,-1) (6) {};
            \draw[color=black!50] (1) -- (2) -- (3) -- (5);
            \draw[color=black!50] (3) -- (4) -- (6);
            \draw[color=black!50] (2) -- (4);
            \node at (0,-3) {$2\sqrt{3}-3$};
        \end{scope};
        \begin{scope}[shift={(-3,1)}];
            \draw[color=black] (0,0) circle [radius=2];
            \draw[fill=blue!10] (-1.5,0) circle [radius=0.5] node[circle, inner sep = 0, minimum size = 3, fill=black] (2) {};
            \draw[fill=blue!10] (-0.5, 0) circle [radius=0.5] node[circle, inner sep = 0, minimum size = 3, fill=black] (3) {};
            \draw[fill=blue!10] (0.5, 0) circle [radius=0.5] node[circle, inner sep = 0, minimum size = 3, fill=black] (4) {};
            \draw[fill=blue!10] (1.5, 0) circle [radius=0.5] node[circle, inner sep = 0, minimum size = 3, fill=black] (5) {};
            \node[circle, inner sep = 0, minimum size = 3, fill=black] at (-2,0) (1) {};
            \node[circle, inner sep = 0, minimum size = 3, fill=black] at (2,0) (6) {};
            \draw[color=black!50] (1) -- (2) -- (3) -- (4) -- (5) -- (6);
            \node at (0,-3) {$\frac{1}{4}$};
        \end{scope};
        \begin{scope}[shift={(2,1)}];
            \draw[color=black] (0,0) circle [radius=2];
            \draw[fill=blue!10] (0, 1.33) circle [radius=0.67] node[circle, inner sep = 0, minimum size = 3, fill=black] (2) {};
            \draw[fill=blue!10] (0,0) circle [radius=0.67] node[circle, inner sep = 0, minimum size = 3, fill=black] (3) {};
            \draw[fill=blue!10] (1.15, -0.67) circle [radius=0.67] node[circle, inner sep = 0, minimum size = 3, fill=black] (4) {};
            \draw[fill=blue!10] (-1.15, -0.67) circle [radius=0.67] node[circle, inner sep = 0, minimum size = 3, fill=black] (6) {};
            \node[circle, inner sep = 0, minimum size = 3, fill=black] at (0,2) (1) {};
            \node[circle, inner sep = 0, minimum size = 3, fill=black] at (1.73,-1) (5) {};
            \node[circle, inner sep = 0, minimum size = 3, fill=black] at (-1.73,-1) (7) {};
            \draw[color=black!50] (1) -- (2) -- (3) -- (4) -- (5);
            \draw[color=black!50] (3) -- (6) -- (7);
            \node at (0,-3) {$\frac{1}{3}$};
        \end{scope};
        \begin{scope}[shift={(7,1)}];
            \draw[color=black] (0,0) circle [radius=2];
            \draw[fill=blue!10] (0.83, 0.83) circle [radius=0.83] node[circle, inner sep = 0, minimum size = 3, fill=black] (2) {};
            \draw[fill=blue!10] (0.83, -0.83) circle [radius=0.83] node[circle, inner sep = 0, minimum size = 3, fill=black] (3) {};
            \draw[fill=blue!10] (-0.83, 0.83) circle [radius=0.83] node[circle, inner sep = 0, minimum size = 3, fill=black] (4) {};
            \draw[fill=blue!10] (-0.83, -0.83) circle [radius=0.83] node[circle, inner sep = 0, minimum size = 3, fill=black] (6) {};
            \node[circle, inner sep = 0, minimum size = 3, fill=black] at (1.41, 1.41) (1) {};
            \node[circle, inner sep = 0, minimum size = 3, fill=black] at (1.41, -1.41) (5) {};
            \node[circle, inner sep = 0, minimum size = 3, fill=black] at (-1.41, 1.41) (7) {};
            \node[circle, inner sep = 0, minimum size = 3, fill=black] at (-1.41, -1.41) (8) {};
            \draw[color=black!50] (1) -- (2) -- (3) -- (5);
            \draw[color=black!50] (3) -- (6) -- (8);
            \draw[color=black!50] (6) -- (4) -- (7);
            \draw[color=black!50] (4) -- (2);
            \node at (0,-3) {$\sqrt{2}-1$};
        \end{scope};
        \begin{scope}[shift={(-3, -5)}]; 
            \draw[color=black] (0,0) circle [radius=2];
            \draw[fill=blue!10] (-1.6,0) circle [radius=0.4] node[circle, inner sep = 0, minimum size = 3, fill=black] (2) {};
            \draw[fill=blue!10] (-0.8, 0) circle [radius=0.4] node[circle, inner sep = 0, minimum size = 3, fill=black] (3) {};
            \draw[fill=blue!10] (0, 0) circle [radius=0.4] node[circle, inner sep = 0, minimum size = 3, fill=black] (4) {};
            \draw[fill=blue!10] (0.8, 0) circle [radius=0.4] node[circle, inner sep = 0, minimum size = 3, fill=black] (5) {};
            \draw[fill=blue!10] (1.6, 0) circle [radius=0.4] node[circle, inner sep = 0, minimum size = 3, fill=black] (6) {};
            \node[circle, inner sep = 0, minimum size = 3, fill=black] at (-2,0) (1) {};
            \node[circle, inner sep = 0, minimum size = 3, fill=black] at (2,0) (7) {};
            \draw[color=black!50] (1) -- (2) -- (3) -- (4) -- (5) -- (6) -- (7);
            \node at (0,-3) {$\frac{1}{5}$};
        \end{scope};
        \begin{scope}[shift={(2,-5)}];
            \draw[color=black] (0,0) circle [radius=2];
            \draw[fill=blue!10] (-1.5,0) circle [radius=0.5] node[circle, inner sep = 0, minimum size = 3, fill=black] (2) {};
            \draw[fill=blue!10] (-0.5, 0) circle [radius=0.5] node[circle, inner sep = 0, minimum size = 3, fill=black] (3) {};
            \draw[fill=blue!10] (0.5, 0) circle [radius=0.5] node[circle, inner sep = 0, minimum size = 3, fill=black] (4) {};
            \draw[fill=blue!10] (1.5, 0) circle [radius=0.5] node[circle, inner sep = 0, minimum size = 3, fill=black] (5) {};
            \draw[fill=blue!10] (0.3, 1.3) circle [radius=0.5] node[circle, inner sep = 0, minimum size = 3, fill=black] {};
            \node[circle, inner sep = 0, minimum size = 3, fill=black] at (-2,0) (1) {};
            \node[circle, inner sep = 0, minimum size = 3, fill=black] at (2,0) (6) {};
            \draw[color=black!50] (1) -- (2) -- (3) -- (4) -- (5) -- (6);
            \node at (0,-3) {$\frac{1}{4}$};
        \end{scope};
        \begin{scope}[shift={(7, -5)}];
            \draw[color=black] (0,0) circle [radius=2];
            \draw[fill=blue!10] (0, 1.33) circle [radius=0.67] node[circle, inner sep = 0, minimum size = 3, fill=black] (2) {};
            \draw[fill=blue!10] (0,0) circle [radius=0.67] node[circle, inner sep = 0, minimum size = 3, fill=black] (3) {};
            \draw[fill=blue!10] (1.33,0) circle [radius=0.67] node[circle, inner sep = 0, minimum size = 3, fill=black] (4) {};
            \draw[fill=blue!10] (0, -1.33) circle [radius=0.67] node[circle, inner sep = 0, minimum size = 3, fill=black] (6) {};
            \draw[fill=blue!10] (-1.33,0) circle [radius=0.67] node[circle, inner sep = 0, minimum size = 3, fill=black] (8) {};
            \node[circle, inner sep = 0, minimum size = 3, fill=black] at (0,2) (1) {};
            \node[circle, inner sep = 0, minimum size = 3, fill=black] at (2,0) (5) {};
            \node[circle, inner sep = 0, minimum size = 3, fill=black] at (0,-2) (7) {};
            \node[circle, inner sep = 0, minimum size = 3, fill=black] at (-2,0) (9) {};
            \draw[color=black!50] (1) -- (2) -- (3) -- (4) -- (5);
            \draw[color=black!50] (9) -- (8) -- (3) -- (6) -- (7);
            \node at (0,-3) {$\frac{1}{3}$};
        \end{scope};
        \begin{scope}[shift={(12, -5)}];
            \draw[color=black] (0,0) circle [radius=2];
            \draw[fill=blue!10] (0, 1.26) circle [radius=0.74] node[circle, inner sep = 0, minimum size = 3, fill=black] (1) {};
            \draw[fill=blue!10] (1.2, 0.39) circle [radius=0.74] node[circle, inner sep = 0, minimum size = 3, fill=black] (2) {};
            \draw[fill=blue!10] (0.74, -1.02) circle [radius=0.74] node[circle, inner sep = 0, minimum size = 3, fill=black] (3) {};
            \draw[fill=blue!10] (-0.74, -1.02) circle [radius=0.74] node[circle, inner sep = 0, minimum size = 3, fill=black] (4) {};            
            \draw[fill=blue!10] (-1.2, 0.39) circle [radius=0.74] node[circle, inner sep = 0, minimum size = 3, fill=black] (5) {};
            \node[circle, inner sep = 0, minimum size = 3, fill=black] at (0,2) (6) {};
            \node[circle, inner sep = 0, minimum size = 3, fill=black] at (1.9,0.62) (7) {};
            \node[circle, inner sep = 0, minimum size = 3, fill=black] at (1.18, -1.62) (8) {};
            \node[circle, inner sep = 0, minimum size = 3, fill=black] at (-1.18, -1.62) (9) {};
            \node[circle, inner sep = 0, minimum size = 3, fill=black] at (-1.9,0.62) (10) {};
            \draw[color=black!50] (6) -- (1) -- (2) -- (3) -- (4) -- (5) -- (1);
            \draw[color=black!50] (2) -- (7);
            \draw[color=black!50] (3) -- (8);
            \draw[color=black!50] (4) -- (9);
            \draw[color=black!50] (5) -- (10);
            \node at (0,-3) {$\frac{ \sin \frac{\pi}{5}}{1+ \sin \frac{\pi}{5}}$};
        \end{scope};
        \end{scope};
    \end{tikzpicture}
    \caption{The critical configurations and associated critical radii of $\Cf_n$, $n\in \{1, 2,3,4,5 \}$, with the underlying stress graphs and associated configurations of discs. In most pairs $(n,r)$ shown, the critical configuration is unique up to rotation of the unit disc. The exceptions are $\left(5, \frac{1}{4} \right)$ and $\left(5, \frac{1}{3}\right)$, where the isolated disc may move freely; and $\left(4, \frac{1}{3}\right)$, where the outer discs may move freely, provided that the centres are never contained in an open semicircle centred at the origin.}
    \label{fig: gcrits}
\end{figure}

\section{A new homotopy class beyond the first critical radius} \label{chap: nts}

\subsection{Four discs, and braids of thick strings \label{sec: n=4}}

For $n=4$, understanding the spaces $\Cf_{n,r}$ as $r$ changes is comparatively simple, as there are only three critical radii (see Prop. \ref{prop: crit4}). For $r>\sqrt{2}-1$, the configuration space is empty. For $\frac{1}{3}<r\le \sqrt{2}-1$, the configuration space is homotopy equivalent to $\Cf_{4,\sqrt{2}-1}$, in which there is only one configuration up to rotation of the unit disc and permutation of the four discs, so the configuration space is a disjoint union of six circles. For $r\le \frac{1}{4}$, it is the configuration space of points. Thus, it only remains to check $\frac{1}{4}<r\le \frac{1}{3}$. In this case we will demonstrate the existence of a `thick braid': a non-trivial loop in $\Cf_{4,r}$ which is trivial in $\Cf_4$ when the discs of each configuration are replaced by their centres. This immediately shows that the radius of the discs affects the homotopy-type of the space: for $r\le \frac{1}{4}$, $\Cf_{n,r}$ is a $K(PBr_4,1)$, whereas this is no longer the case for $r>\frac{1}{4}$.

\begin{thm} \label{thm: braid}
    Let $\frac{1}{4}< r \le \frac{1}{3}$. Then $\Cf_{4, r}$ contains a non-trivial loop  which is homotopic to $\sigma_3 \sigma_1^{-1} \sigma_3^{-1} \sigma_1$ in the standard representation of the braid group, and therefore trivial when the disc radii are reduced below $\frac{1}{4}$. 
\end{thm}
A realisation of this loop is depicted in Figure \ref{fig: n=4}. 

A key piece of the proof comes in proving that the given loop is non-contractible. This will rely on the following map (inspired by \cite{alp}) and lemma, which will be used in this section and again in \S\ref{chap: deform}.

\begin{defn} \label{defn: ang}
    Given some $\DD \in \Cf_{n,r}(U)$, and some non-repeating sequence $I=(i_k)$ of integers $1\le i_k \le n$, let $\theta_k \coloneqq \theta_k(\DD)$ be the anticlockwise angle from the vector $(1,0)$ to the vector $x_{i_{k+1}}-x_{i_k}$ for all $k$. Let $\phi_k = \theta_{k}-\theta_{k-1}$ for $2 \le k \le n-1$.
    Then the \emph{angle map} is $\ang_I \colon \Cf_{n,r}(U) \to T^{n-2}, \ \DD \mapsto (\phi_2(\DD), \ldots, \phi_{n-1}(\DD))$. If $I$ is the sequence $1, 2, \ldots, n$, then we write $\ang(\DD) \coloneqq \ang_I(\DD)$.
\end{defn}

We may think of $\phi_k$ as a turning angle: if we walk along a straight path from $x_{i_{k-1}}$ to $x_{i_k}$, and then from $x_{i_k}$ to $x_{i_{k+1}}$, then $\phi_k$ is the angle through which we turn at $x_{i_k}$ (see Fig. \ref{fig: ang}). 

\begin{figure}[htb]
\begin{center}
    \begin{tikzpicture}
        \draw[dashed] (2,-0.66) -- (0,0) node[circle, inner sep = 0, minimum size = 3, fill=black] {} -- (1.2,0);
        \draw[dashed] (-2.12,0.7) -- (-0.7,0.7);
        \draw (-2.12,0.7) node[circle, inner sep = 0, minimum size = 3, fill=black] {} -- (0,0) -- (2,1) node[circle, inner sep = 0, minimum size = 3, fill=black] {};
        \node[] at (-0.25,-0.2) {$x_{i_k}$};
        \node[] at (2,1.3) {$x_{i_{k+1}}$};
        \node[] at (-2.3,0.4) {$x_{i_{k-1}}$};
        \draw[->] (0.8,0) to[out=90, in=-60] (0.7, 0.35);
        \draw[->] (1.6,-0.55) to[out=65, in=-60] (1.5, 0.75);
        \draw[->] (-1.2,0.7) to[out=-90, in=60] (-1.25, 0.45);
        \node[] at (-1,1) {$\theta_{k-1}$};
        \node[] at (1.1,0.25) {$\theta_k$};
        \node[] at (2,0) {$\phi_k$};
    \end{tikzpicture}
\end{center}
\caption{The angles $\theta_k$ and $\phi_k$ used to construct the map $\ang$ in Def. \ref{defn: ang}. Angles are taken to be anticlockwise from the dotted line to the solid line, so $\theta_{k-1}<0$ and $\theta_k, \phi_k>0$ in this diagram.} 
\label{fig: ang}
\end{figure}

\begin{figure}[htbp]
\begin{center}
    \begin{tikzpicture}[node distance={15mm}, thick, main/.style={}]
    	\begin{scope}[scale=0.8];
        \draw[color=black] (0,0) rectangle (4,4);
        \node[shape = cross out, draw = red] (5) at (2,2) {};
        \draw[->, color=blue] (1,2) -- (1,1) -- (2,1); 
        \draw[->, color=blue] (2,1) -- (3,1) -- (3,2); 
        \draw[->, color=blue] (3,2) -- (3,3) -- (2,3);
        \draw[->, color=blue] (2,3) -- (1,3) -- (1,2);
        \node[circle, inner sep = 0, minimum size = 6, draw=blue!60] (10) at (1,1) {};
        \node[circle, inner sep = 0, minimum size = 6, draw=blue!60] (11) at (3,1) {};
        \node[circle, inner sep = 0, minimum size = 6, draw=blue!60] (12) at (1,3) {};
        \node[circle, inner sep = 0, minimum size = 6, draw=blue!60] (13) at (3,3) {};
        \draw[color=green] (0.9,1.1) -- (-2.8,-0.1);
        \draw[color=green] (1.1,0.9) -- (-0.1,-2.8);
        \begin{scope}[rotate=-90, shift={(-4,0)}];
            \draw[color=green] (0.9,1.1) -- (-2.8,-0.1);
            \draw[color=green] (1.1,0.9) -- (-0.1,-2.8);
        \end{scope};
        \begin{scope}[rotate=90, shift={(0,-4)}];
            \draw[color=green] (0.9,1.1) -- (-2.8,-0.1);
            \draw[color=green] (1.1,0.9) -- (-0.1,-2.8);
        \end{scope};
        \begin{scope}[rotate=180, shift={(-4,-4)}];
            \draw[color=green] (0.9,1.1) -- (-2.8,-0.1);
            \draw[color=green] (1.1,0.9) -- (-0.1,-2.8);
        \end{scope};
        \begin{scope}[shift={(-2,-2)}, rotate=-90];
            \draw[color=black] (0,0) circle [radius=2];
            \draw[color=black, fill=blue!10] (0,0.6) circle [radius=0.58] node {4};
            \draw[color=black, fill=blue!10] (0,-0.6) circle [radius=0.58] node {1};
            \draw[color=black, fill=blue!10] (-1.2,0.6) circle [radius=0.58] node {3};
            \draw[color=black, fill=blue!10] (-1.2,-0.6) circle [radius=0.58] node {2};
            \draw[color=black!60, densely dashed] (-0.6,0) -- (0,0) to [out=0, in=-90] (0.6,0.6) to[out=90, in=0] (0,1.2) -- (-1.2,1.2) to[out=180, in=90] (-1.8, 0.6) -- (-1.8, -0.6) to[out=-90, in=180] (-1.2,-1.2) to[out=0, in=-90] (-0.6,-0.6) -- (-0.6,0);
            \draw[->, color=black!60] (0, -1.25) to[out=-100, in=-70] (-1, -1.3);
            \draw[->, color=black!60] (-0.6, 1.25) to[out=80, in=110] (0.6, 1.3);
        \end{scope};
        \begin{scope}[shift={(-2,6)}, rotate=-90];
            \draw[color=black] (0,0) circle [radius=2];
            \draw[color=black, fill=blue!10] (-1.2,0.6) circle [radius=0.58] node {4};
            \draw[color=black, fill=blue!10] (1.2,-0.6) circle [radius=0.58] node {1};
            \draw[color=black, fill=blue!10] (0,0.6) circle [radius=0.58] node {3};
            \draw[color=black, fill=blue!10] (0,-0.6) circle [radius=0.58] node {2};
            \draw[color=black!60, densely dashed] (0.6,0) -- (1.2,0) to [out=0, in=90] (1.8,-0.6) to[out=-90, in=0] (1.2,-1.2) -- (0,-1.2) to[out=180, in=-90] (-0.6, -0.6) -- (-0.6, 0.6) to[out=90, in=180] (0,1.2) to[out=0, in=90] (0.6,0.6) -- (0.6,0);
            \draw[->, color=black!60] (0.6, -1.25) to[out=-100, in=-70] (-0.5, -1.3);
            \draw[->, color=black!60] (-1.2, 1.25) to[out=80, in=110] (0, 1.3);
        \end{scope};
        \begin{scope}[shift={(6,-2)}, rotate=-90];
            \draw[color=black] (0,0) circle [radius=2];
            \draw[color=black, fill=blue!10] (1.2,0.6) circle [radius=0.58] node {4};
            \draw[color=black, fill=blue!10] (-1.2,-0.6) circle [radius=0.58] node {1};
            \draw[color=black, fill=blue!10] (0,0.6) circle [radius=0.58] node {3};
            \draw[color=black, fill=blue!10] (0,-0.6) circle [radius=0.58] node {2};
            \draw[color=black!60, densely dashed] (-0.6,0) -- (-0.6,0.6) to [out=90, in=180] (0,1.2) to[out=0, in=90] (0.6,0.6) -- (0.6,-0.6) to[out=-90, in=0] (0, -1.2) -- (-1.2, -1.2) to[out=180, in=-90] (-1.8,-0.6) to[out=90, in=180] (-1.2,0) -- (-0.6,0);
            \draw[->, color=black!60] (-0.6, -1.25) to[out=-80, in=-110] (0.5, -1.3);
            \draw[->, color=black!60] (1.2, 1.25) to[out=100, in=70] (0.1, 1.3);
        \end{scope};
        \begin{scope}[shift={(6,6)}, rotate=-90];
            \draw[color=black] (0,0) circle [radius=2];
            \draw[color=black, fill=blue!10] (0,0.6) circle [radius=0.58] node {4};
            \draw[color=black, fill=blue!10] (0,-0.6) circle [radius=0.58] node {1};
            \draw[color=black, fill=blue!10] (1.2,0.6) circle [radius=0.58] node {3};
            \draw[color=black, fill=blue!10] (1.2,-0.6) circle [radius=0.58] node {2};
            \draw[color=black!60, densely dashed] (0.6,0) -- (0,0) to [out=180, in=-90] (-0.6,0.6) to[out=90, in=180] (0,1.2) -- (1.2,1.2) to[out=0, in=90] (1.8, 0.6) -- (1.8, -0.6) to[out=-90, in=0] (1.2,-1.2) to[out=180, in=-90] (0.6,-0.6) -- (0.6,0);
            \draw[->, color=black!60] (0, -1.25) to[out=-80, in=-110] (1, -1.3);
            \draw[->, color=black!60] (0.6, 1.25) to[out=100, in=70] (-0.6, 1.3);
        \end{scope};
        \node[color=black] (10) at (0, -0.3) {\footnotesize{$(-\pi,-\pi)$}};
        \node[color=black] (11) at (4, 4.3) {\footnotesize{$(\pi,\pi)$}};
        \draw[->] (-2.8,3.7) to[out=-120, in=120] (-2.8,0.3);
        \draw[->] (6.8,0.3) to[out=60, in=-60] (6.8,3.7);
        \draw[->] (0.3,-2.8) to[out=-30, in=-150] (3.7,-2.8);
        \draw[->] (3.7,6.8) to[out=150, in=30] (0.3,6.8);
        \node (12) at (2,-3.8) {$\sigma_1^{-1}$};
        \node (13) at (2, 7.7) {$\sigma_1$};
        \node (14) at (-3.8, 2) {$\sigma_3^{-1}$};
        \node (15) at (7.8, 2) {$\sigma_3$};
        \end{scope};
    \end{tikzpicture}
\end{center}
\caption{A loop of configurations of four discs of radius 0.29 inside the unit disc, starting at the top right, which is homotopic to the sequence $\sigma_3 \sigma_1^{-1} \sigma_3^{-1} \sigma_1$ in the standard presentation of the braid group. This descends to the loop $\partial \left[-\frac{\pi}{2}, \frac{\pi}{2}\right]^2 \subset T^2$ (blue) under the map $\ang$. The blue loop is non-contractible in $T^2 \backslash \{0\}$; then since $0\notin \ang(\Cf_{4, 0.29})$, the loop of configurations is also non-contractible. This homotopy class persists for $\frac{1}{4}< r \le \frac{1}{3}$.} 
\label{fig: n=4}
\end{figure}

\begin{lem} \label{lem: punc}
    Let $n,k\in \NN$, $r>\frac{1}{k}$ and $U\subset \RR^2$ a connected, codimension zero submanifold such that any disc of radius $r$ in $U$ is contained in $D^2$. Let $T$ be some $m$-dimensional manifold, $f\colon \Cf_{n,r}(U) \to T$ be a continuous map, and $p\in T$ be such that if $f(\mathcal{D})=p$, then some $k$ discs in $\DD$ lie on a straight line. Let $S$ be an $(m-1)$-sphere in $T\backslash \{p\}$ such that $[S] \neq 0 \in \pi_{m-1}(T\backslash \{p\})$ and $\mathcal{S}\colon S\to \Cf_{n,r}(U)$ a local section of $f$. Then $\mathcal{S}$ represents a non-trivial class in $\pi_{m-1}(\Cf_{n,r}(U))$.
\end{lem}
\begin{proof}
    First, we note that $\mathcal{S}$ is an $(m-1)$-sphere inside $\Cf_{n,r}(U)$, so represents some class in $\pi_{m-1}(\Cf_{n,r}(U))$. Since $r>\frac{1}{k}$, we cannot fit $k$ discs of radius $r$ on one straight line inside $D^2$, nor therefore inside $U$, so $f(\Cf_{n,r}(U)) \subset T\backslash \{p\}$. Given that $S = f(\mathcal{S})$, we have $f_*[\SS] = [S] \neq 0 \in \pi_{m-1}(T\backslash \{p\})$, and the result follows since $f_*$ is a homomorphism.
\end{proof}

Once a choice of $S, \SS$ is made, this lemma can be used to prove that $\SS$ is non-trivial in its homotopy group. We will typically apply this lemma with $f=\ang_I$ for some appropriate choice of $I$, $T=T^{n-2}$, and $p=0$, since $p=0$ is often missing from the image of $\ang$ if the disc radius is large enough.
 
\begin{proof}[Proof of Theorem \ref{thm: braid}]
    Consider the loop $S = \partial \left( \left[-\frac{\pi}{2}, \frac{\pi}{2}\right]^2 \right) \subset T^2$. 
    We can then define a local section $\mathcal{S} \colon S \to \Cf_{4,r}$ for any $r\le \frac{1}{2+\sqrt{2}}$ by
    \[ (\phi_2, \phi_3) \mapsto \frac{1}{2+\sqrt{2}} \left( 
    \begin{array}{l}
           (-1-\cos\phi_2-\cos\phi_3, \sin\phi_2-\sin\phi_3) \\
           (-1+\cos\phi_2-\cos\phi_3, -\sin\phi_2-\sin\phi_3) \\ 
           (1+\cos\phi_2-\cos\phi_3, -\sin\phi_2-\sin\phi_3) \\
           (1+\cos\phi_2+\cos\phi_3, -\sin\phi_2+\sin\phi_3)     
    \end{array} \right)      
    \]
    where the rows are the coordinates of the centres, denoted $x_i$ ($i\in \{1,2,3,4\}$), of the four discs. 
     
    It is possible from the given coordinates to check $\SS$ is well-defined. First note that on $S$, there is always one $\phi_j = \pm \frac{\pi}{2}$. Then $\sin\phi_j = \pm 1$ and $\cos\phi_j=0$, so letting $\phi$ represent the other coordinate, $|x_i|^2 = \left(\frac{1}{2+\sqrt{2}}\right)^2 \left((1\pm \cos\phi)^2 + (1\pm \sin\phi)^2\right) \le \left(\frac{1+\sqrt{2}}{2+\sqrt{2}}\right)^2 \le (1-r)^2$. Therefore, $D_i\subset D^2$ for each $i$. We may check similarly that any two disc centres are separated by a distance at least $2r$. 
     
    By construction, discs 2 and 3 lie on the same horizontal line at all values of $\SS$, while the coordinates of discs 1 and 4 are chosen to satisfy $\ang\circ\SS = \mathrm{id}_S$, so this is indeed a local section to $\ang$. 

    Consider $\{\SS(\phi_2, \phi_3) \mid (\phi_2, \phi_3) \in S\}$. At each corner of $S$, the four disc centres of $\SS(\phi_1, \phi_2)$ form a parallelogram. Along the first edge, $\{(\frac{\pi}{2}- t\pi, \frac{\pi}{2}) \mid t\in [0,1]\}$, we can check from the coordinates that $D_1$ moves half a turn anticlockwise around $D_2$, while the three other discs stay in the same position relative to one another. This causes $D_1$ and $D_2$ to switch positions in the parallelogram; therefore the underlying path in $\Cf_{4}$ traced out by the disc centres along this side is homotopic to the braid $\sigma_1$. Similarly, the remaining sides of $S$ yield the braids $\sigma_3^{-1}$, $\sigma_1^{-1}$, $\sigma_3$. 

    If $\ang(\DD) = 0$, then the four discs lie on a straight line. Furthermore, $[S] \neq 0 \in \pi_{1}(T^2\backslash \{0\})$. Therefore $\mathcal{S}$ represents a non-trivial class in $\pi_1(\Cf_{4,r})$ by Lemma \ref{lem: punc}.
     
    Finally, this homotopy class persists until the second critical radius by Fact \ref{fact: crits}, which is $r=\frac{1}{3}$ (see Prop. \ref{prop: crit4}).
\end{proof}

\subsection{More discs yields a higher-dimensional homotopy class \label{sec: extend}}

It is well-known \cite[Cor. 2.2]{aspher} that the configuration space of $n$ points in any codimension-zero subset of the plane (and in particular the open unit disc) is aspherical -- that is, its homotopy groups in dimension 2 or greater are trivial. In contrast, here we generalise the work of \S\ref{sec: n=4} to demonstrate a non-trivial element of $\pi_{n-3}(\Cf_{n,\frac{1}{n}+\varepsilon})$ for all $n\ge 5$, when $\varepsilon>0$ is sufficiently small. 

\begin{thm} \label{thm: nts}
There is a non-trivial element in $\pi_{n-3}(\Cf_{n,r})$ for all $n\ge 5$ and $\frac{1}{n}<r \le \frac{1}{n-1}$. 
\end{thm}

To prove this theorem, we will construct a sphere of configurations for sufficiently small $r>\frac{1}{n}$. The following lemma is used to prove that the configurations in the constructed sphere lie inside $D^2$.

\begin{lem} \label{lem: fit}
Let $\DD = (D_1, D_2, \ldots, D_n)$ be a configuration of open unit discs such that $D_k$ is in contact with $D_{k+1}$ for all $1\le k\le n-1$ and $\phi_l=\pm\xi$ for some $2\le l\le n-1$, $0<\xi\le \frac{\pi}{2}$. Then $\bigcup \DD \subset B\left(\frac{x_1+x_n}{2}, n-4\sin^2\frac{\xi}{4}\right)$.
\end{lem}
\begin{proof}
    Let $z=\frac{x_1+x_n}{2}$. The claim is equivalent to proving $x_i\in B\left(z, n-4\sin^2\frac{\xi}{4}-1\right)$, or equivalently $|z-x_i| \le n - 4\sin^2\frac{\xi}{4} - 1 = n - 3 + 2\cos\frac{\xi}{2}$, for all $i$. We have two cases: 
    
    $\bm{i\neq l}$: By the triangle inequality, $|z-x_i| \le \frac{1}{2} \left(|x_1-x_i|+|x_i-x_n| \right)$. Then, assuming (by reversing the labelling if necessary) that $i>l$, we have
    \begin{IEEEeqnarray*}{rCl}
    |x_1-x_i| &\le& \sum_{k=1}^{l-2}|x_k-x_{k+1}| + |x_{l-1}-x_{l+1}| + \sum_{k=l+1}^{i-1}|x_k-x_{k+1}| \\
    &=& 2(l-2) + 4\cos\frac{\xi}{2} + 2(i-1-l) \\
    &=& 2(i-3) + 4\cos\frac{\xi}{2} \\
    |x_i-x_n| &\le& \sum_{k=i}^{n-1}|x_k-x_{k+1}| \\
    &=& 2(n-i)
    \end{IEEEeqnarray*}
    Thus $|z-x_i| \le n-3+2\cos\frac{\xi}{2}$, as required. 
    
    $\bm{i=l}$: First, consider our configuration placed in the plane with $x_{l-1}$ and $x_{l+1}$ equidistant from the origin along the $x$-axis. Then (up to reflection) $x_{l-1} = (-c, 0)$, $x_{l} = (0, -s)$, $x_{l+1} = (c, 0)$, where $c=2\cos\frac{\xi}{2}$ and $s=2\sin\frac{\xi}{2}$. As in the first case, we have $|x_1-x_{l-1}| \le \sum_{i=1}^{l-2} |x_i - x_{i+1}| = 2(l-2)$, $|x_n-x_{l+1}| \le \sum_{i=l+1}^{n-1} |x_i - x_{i+1}| = 2(n-l-1)$, so there are $A\in [0,l-2]$, $B\in [0,n-l-1]$, and angles $\alpha, \beta$ such that $x_1 = (-c-2A\sin\alpha, 2A\cos\alpha)$ and $x_n = (c+2B\sin\beta, 2B\cos\beta)$. Then $z - x_l = (B\sin\beta - A\sin\alpha, A\cos\alpha + B\cos\beta + s)$. Therefore
    \begin{IEEEeqnarray*}{rCl}
        |z-x_l|^2 &=& (B\sin\beta - A\sin\alpha)^2 + (A\cos\alpha + B\cos\beta + s)^2 \\
        &=& B^2\sin^2\beta - 2AB\sin\alpha\sin\beta + A^2\sin^2\alpha \\ 
        && + \ B^2\cos^2\beta + 2AB\cos\alpha\cos\beta + A^2\cos^2\alpha + s^2 + 2sA\cos\alpha + 2sB\cos\beta \\
        &=& B^2 + 2AB\cos(\alpha+\beta) + A^2 + s^2 + 2sA\cos\alpha + 2sB\cos\beta \\
        &\le& B^2 + 2AB + A^2 + s^2 + 2sA + 2sB \\
        &=& (A+B+s)^2 \\
        &\le& \left(n-3+2\sin\frac{\xi}{2}\right)^2
    \end{IEEEeqnarray*}
    We complete our proof by noting that $\sin\frac{\xi}{2} \le \cos\frac{\xi}{2}$. 
\end{proof}
\begin{rem} \label{rem: sharp}
This maximum is sharp, as it is attained by discs $1$ and $n$ in the configuration in which all discs except $l$ are collinear.
\end{rem}

\begin{proof}[Proof of Theorem \ref{thm: nts}]
First, consider the case $\frac{1}{n} < r \le \frac{1}{n-4\sin^2(\frac{\pi}{2n})}$. We take the $(n-3)$-sphere $S=\partial\left[ -\frac{2\pi}{n}, \frac{2\pi}{n} \right]^{n-2}$, and claim that the following algorithm defines a local section $\mathcal{S}\colon S \to \Cf_{n,r}$. We construct the configuration $(D_1, \ldots, D_n) \coloneqq \mathcal{S}(\phi_2, \ldots, \phi_{n-1})$ as follows:
\begin{enumerate}
    \item Place $D_1$ centred at the origin of the plane and $D_2$ in contact with it, centred at $(2r, 0)$.
    \item Place each subsequent disc $D_i$ in contact with $D_{i-1}$, such that $x_i$ lies on the ray from $x_{i-1}$ at angle $\phi_{i-1}$.
    \item Translate all discs by $-\frac{x_1+x_n}{2}$.
\end{enumerate}
\begin{figure}[htbp]
    \begin{center}
        \begin{tikzpicture}
        \begin{scope}[scale=1.5];
            \draw[color=black] (0,0.3) circle [radius=0.5] node[circle, inner sep = 0, minimum size = 6, fill=black, label=left:1] {};
            \draw[color=black] (1,0.3) circle [radius=0.5] node[circle, inner sep = 0, minimum size = 6, fill=black, label=above:2] {};
            \node[rotate=18] at (1.8,0.54) {$\ldots$};
            \draw[color=black] (2.68,0.99) circle [radius=0.5] node[circle, inner sep = 0, minimum size = 6, fill=black, label=above:$i-2$] {};
            \draw[color=black] (3.66,1.17) circle [radius=0.5] node[circle, inner sep = 0, minimum size = 6, fill=black, label=above:$i-1$] {};
            \draw[color=black, fill=black!10] (4.6,0.82) circle [radius=0.5] node[circle, inner sep = 0, minimum size = 6, fill=black, label=right:$i$] {};  
            \draw (0,0.3) -- (1,0.3) -- (1.6,0.48);
            \draw (2.02, 0.65) -- (2.68,0.99) -- (3.66,1.17) -- (4.6,0.82);
            \draw[densely dashed] (3.66,1.17) -- (5.52, 1.53);
            \draw[very thick,domain=-20.4:10.4]  plot ({3.66+0.3*cos(\x)}, {1.17+0.3*sin(\x)}) node[inner sep = 0, below left= 0.25 and -0.2] {$\phi_{i-1}$};
            \end{scope};
        \end{tikzpicture}
    \end{center}
        \caption{The placement of the $i$-th disc in the construction of the configurations of Theorem \ref{thm: nts}.}
        \label{fig: disc i}
    \end{figure}
    
$\ang\circ\SS = \mathrm{id}_S$ by construction (step 2). The placement of $D_i$ at step 2 is continuous with respect to $\phi_i$, and the translation in step 3 is continuous with respect to $x_1$ and $x_n$, so $\SS$ is continuous. Thus, we prove our claim if we show that $\SS(\Phi) \in \Cf_{n,r}$ for all $\Phi\in S$. 

At the placement of $D_i$ in step 2, consider the two regular $n$-gons which have $\overline{x_{i-1}x_i}$ as one of their sides. Since $|\phi_k|\le \frac{2\pi}{n}$, the external angle of a regular $n$-gon, for all $k$, the sequence of edges \[x_1 \longrightarrow x_2 \longrightarrow \ldots \longrightarrow x_{i-1} \longrightarrow x_i \] lies outside or on the boundary of these two $n$-gons. Thus, since $i-1<n$, $D_i\cap D_j = \emptyset$ for $j\in\{1,2,\ldots,i-1\}$. Therefore no two discs in $\SS(\phi_2, \ldots, \phi_{n-1})$ overlap. Since every point in $S$ has at least one coordinate equal to $\pm \frac{2\pi}{n}$ and $D_{i+1}$ touches $D_i$ for all $i$, we can apply Lemma \ref{lem: fit} with $\xi=\frac{2\pi}{n}$ to show that $\bigcup \SS(\phi_2, \ldots, \phi_{n-1}) \subset B\left(0, \left(n-4\sin^2\frac{\pi}{2n}\right)r\right) \subset D^2$, which completes the proof of our claim. 

Next, we note that if $\ang(\DD)=0$, then the $n$ discs lie on one straight line. Moreover, $[S] \neq 0 \in \pi_{n-3}(T^{n-2}\backslash \{0\})$. Therefore we may apply Lemma \ref{lem: punc} to show that $\SS$ represents a non-trivial class in $\pi_{n-3}(\Cf_{n,r})$.

Finally, this non-trivial class persists up to the second critical radius by Fact \ref{fact: crits}. This is $\frac{1}{n-1}$ -- when $n\in \{4,5\}$, this is Prop. \ref{prop: crit4} and \ref{prop: crit5}, while for $n\ge 6$, we note that $\frac{3}{2n+3} \ge \frac{1}{n-1}$, and so this must be the second critical radius by Fact \ref{fact: critnd}.
\end{proof}

\begin{rem}
    While we are able to extend the existence of a homotopic sphere in $\Cf_{n,r}$ to radii up to $r=\frac{1}{n-1}$ in the proof of Thm. \ref{thm: nts}, Remark \ref{rem: sharp} shows this will not project to $\partial \left( [-\frac{2\pi}{n}, \frac{2\pi}{n}]^{n-2}\right)$ for radii above $\frac{1}{n-4\sin^2(\frac{\pi}{2n})}$: for these radii, any configuration with angles $(0, \ldots, 0, -\frac{\pi}{n}, \frac{2\pi}{n}, -\frac{\pi}{n}, 0, \ldots, 0)$ has diameter greater than 2, so does not lie in the unit disc.
\end{rem}
\begin{rem} \label{rem: coord}
    This algorithm allows us to construct explicit coordinates for our configuration in terms of $\theta_i = \sum_{j=1}^i \phi_j$. At step 2, $x_i = 2r(\sum_{j=1}^{i-1} \cos\theta_j, \sum_{j=1}^{i-1} \sin\theta_j)$ for $i\ge 2$, with $\theta_1 = 0$. After translation by $-\frac{1}{2}(x_1+x_n) = -r(\sum_{j=1}^{n-1} \cos\theta_j, \sum_{j=1}^{n-1} \sin\theta_j)$, we have $x_i = r (\sum_{j=1}^{i-1} \cos\theta_j - \sum_{j=i}^{n-1} \cos\theta_j, \sum_{j=1}^{i-1} \sin\theta_j - \sum_{j=i}^{n-1} \sin\theta_j)$.
\end{rem}
\begin{rem} \label{rem: thin}
    We can construct a sphere of the same homotopy class by lifting $S=\partial \left( [-\xi, \xi]^{n-2}\right)$ to $\Cf_{n,r}$ by the same algorithm for all $0<\xi\le \frac{2\pi}{n}$ and $\frac{1}{n} < r \le \frac{1}{n-4\sin^2(\frac{\xi}{4})}$. Then, by Remark \ref{rem: coord}, $x_i \to \left( \frac{2i-n-1}{n}, 0 \right)$ as $\xi \to 0$ for all configurations of the sphere. That is to say, the configurations of the sphere are constrained increasingly close to the critical configuration.
\end{rem}
\begin{rem} \label{rem: more}
    This construction gives rise to many distinct homotopy classes in  $\pi_{n-3}(\Cf_{n,r})$ for $\frac{1}{n}< r \le \frac{1}{n-1}$. Given $r \le \frac{1}{n-4\sin^2 \frac{\pi}{2n}}$, take $\perm\in S_n$, and let $\SS'=\perm\cdot\SS$, where the action of $\perm$ on $\Cf_{n,r}$ permutes the discs in the natural way. It can be shown that there is a path $\gamma\colon [0,1]\to \Cf_{n,r}$ such that $\gamma(0) \in \SS$ and $\gamma(1) = \perm\cdot \gamma(0)$, so we can consider $[\SS']$ in the same homotopy group as $[\SS]$. Then $\ang(\SS')$ encircles some $0\neq \Phi \in \{(\phi_2, \ldots, \phi_{n-1})\in T^{n-2} \mid \forall_i \ \phi_i\in \{0,\pi\}\}$, where $\Phi$ corresponds to $n$ discs lying on a line. Hence $\ang_*([\SS']) \neq [S]$, so $[\SS'] \neq [\SS]$.
\end{rem}

\section{The full homotopy type for four discs} \label{chap: 4}

In \S\ref{sec: n=4}, we determine the homotopy type of $\Cf_{4,r}$ for $r\le \frac{1}{4}$ and $r> \frac{1}{3}$. We also show that there exists a non-trivial `thick braid' in $\pi_1(\Cf_{4, r})$ for $\frac{1}{4} < r \le \frac{1}{3}$ which vanishes under the inclusion $\Cf_{n,r} \to \Cf_n$. In this section, we develop a full picture of the homotopy type of $\Cf_{4,r}$ when $\frac{1}{4}<r \le\frac{1}{3}$. We start by noting that $\Cf_{4,r} \simeq \Cf_{4, \frac{1}{3}} = \tau^{-1}[\frac{1}{3}, \infty)$ by Fact \ref{fact: crits}. 

Let $Y$ be the 1-dimensional cell complex consisting of
\begin{itemize}
        \item a set, $V_1$, of 0-cells, indexed over the 6 vertices of an octahedron,
        \item a set, $V_2$, of 0-cells, indexed over the 8 faces of an octahedron, and
        \item a set, $E$, of 1-cells, in which $v_1 \in V_1$ and $v_2 \in V_2$ bound a 1-cell if and only if the vertex corresponding to $v_1$ is on the boundary of the face corresponding to $v_2$.
\end{itemize} 
We will prove the following.

\begin{thm} \label{thm: graph2}
    Let $r\in \left( \frac{1}{4}, \frac{1}{3} \right]$. Then $\Cf_{4, r} \simeq Y\times S^1$.
\end{thm}

We will prove that $\Cf_{4, \frac{1}{3}}$ is a trivial circle bundle (a topological space of the form $X\times S^1$ for some space $X$) in \S\ref{sec: triv}. Then, we will show how the quotient $\Cf_{4, \frac{1}{3}}/S^1$ sits naturally on the surface of an octahedron in \S\ref{sec: struc}. Finally, we will retract $\Cf_{4, \frac{1}{3}} / S^1$ onto $Y$ in \S\ref{sec: contr}. 

Throughout this section, we will denote an ordering of $m$ numbers from the set $\{1, 2, \ldots, n\}$ by $(i_1\ i_2\ \ldots\ i_m)$, and a cyclic ordering by $[i_1\ i_2\ \ldots\ i_m]$, where $i_1, i_2, \ldots, i_m$ are distinct elements of $\{1, 2, \ldots, n\}$. Then for a cyclic ordering $\perm$, denote  $\hat{\perm}_{i_k} \coloneqq (i_{k+1}\ \ldots\ i_m\ i_1\ \ldots\ i_{k-1})$ for any $1\le k\le m$. For an ordering $\rho$, we will denote the corresponding cyclic ordering (where we forget the start point) by $[\rho]$. Furthermore, we will sometimes borrow from the language of the permutation group and write $i_{k+1} = \perm(i_k)$ for $1\le k \le m-1$ (and $i_1 = \perm(i_m)$ in the case that $\perm$ is a cyclic ordering).

We also, given distinct points $P,Q,R$, denote by $\angle PQR$ the anti-clockwise angle (taken in $[0,2\pi)$) from the line segment $\overline{QP}$ to the line segment $\overline{QR}$ about $Q$. This means that $\angle PQR + \angle RQP = 2\pi$. Sometimes, we will be interested in the smaller of the two angles $\angle PQR, \angle RQP$, in which case we will write $|\angle PQR| = \min \{ \angle PQR, \angle RQP \}$.

The proof will also require the following claims about configurations of discs.

\begin{lem} \label{lem: flow dist}
    Take $R\in(0,\infty)$, and let $x,y\in \overline{B(0,R)}$ be such that $|x-y|\ge R$. Then for all $\lambda, \mu \in \left[1, \infty \right)$, we have $|\lambda x - \mu y| \ge R$. 
\end{lem}
\begin{proof}
    The hypotheses of the lemma are equivalent to the claim $x\in \overline{B(0,R)} \backslash B(y,R)$. Take $t\in [1, \lambda]$. Since $x$ lies on the line segment joining 0 and $\lambda x$, with $0\in \overline{B(y,R)}$ and $x\notin B(y,R)$, it follows that $tx \notin B(y,R)$ by convexity of the ball. Therefore $y\notin \bigcup_{t\in[1, \lambda]} B(tx, R)$. Since $0\in \bigcup_{t\in[1, \lambda]} B(tx, R)$, it follows by convexity that $\mu y \notin \bigcup_{t\in[1, \lambda]} B(tx, R)$; in particular, $|\lambda x - \mu y| \ge R$.
\end{proof}

\begin{lem} \label{lem: angle}
    Let $\DD \in \Cf_{n, \frac{1}{3}}$ for some $n\in\NN$, $x,y\ne 0$ be the centres of some discs in $\DD$. Then $\angle x0y \ge \frac{\pi}{3}$.
\end{lem}
\begin{proof}
    Let $x,y \neq 0$ be the centres of two discs in $\DD$. Then $|x-y| \ge \frac{2}{3}$ and $x,y \in \overline{B \left( 0, \frac{2}{3} \right)}$. Let $\hat{x} = \frac{2}{3|x|}x$, $\hat{y} = \frac{2}{3|y|}y$, from which $|\hat{x}| = |\hat{y}| = \frac{2}{3}$ and $\angle \hat{x} 0 \hat{y} = \angle x0y$. Then by Lemma \ref{lem: flow dist}, we have $\left| \hat{x} - \hat{y} \right| \ge \frac{2}{3}$. At equality, we have the triangle with vertices  $0, \hat{x}, \hat{y}$ is equilateral; thus $\angle \hat{x} 0 \hat{y} \ge \frac{\pi}{3}$.
\end{proof}

As one disc approaches the origin, however, we can strengthen this claim to show that the angle of separation to another disc must be at least $\frac{\pi}{2}$ in the limit. 

\begin{lem} \label{lem: angle at zero}
    Let $\DD \in \Cf_{n, \frac{1}{3}}$ for some $n\in\NN$, and let $x,y \ne 0$ be the centres of some discs in $\DD$. Then $\angle x0y \ge \arccos \left( \frac{3}{4}|x| \right)$. 
\end{lem}
\begin{proof}
    Let $\theta = \angle x0y$ be the angle subtended at the origin by $x$ and $y$. Denote $a=|x|, b=|y|, c=|x-y|$. Then $c^2 = a^2+b^2-2ab\cos(\theta)$ by the cosine rule. We further note $c^2 \ge \frac{4}{9}$.
    
    By Lemma \ref{lem: angle}, we have $\frac{\pi}{3} \le \theta \le \frac{5\pi}{3}$, from which we see that $2a \cos( \theta) \le a \le \frac{2}{3} + b$. Now,
    \begin{IEEEeqnarray*}{rClCrCl}
        2a \cos( \theta) & \le & \frac{2}{3} + b &
        \Rightarrow & 2a \left( \frac{2}{3} - b \right) \cos( \theta) & \le & \frac{4}{9} - b^2 \\
        &&& \Rightarrow & b^2 - 2ab \cos( \theta) & \le & \frac{4}{9} - \frac{4}{3}a \cos \theta \\
        &&& \Rightarrow & a^2 + b^2 - 2ab \cos( \theta) & \le & a^2 + \frac{4}{9} - \frac{4}{3}a \cos \theta \quad .
    \end{IEEEeqnarray*}
    Therefore, we have shown that \[ \frac{4}{9} \le a^2 + \frac{4}{9} - \frac{4}{3}a \cos \theta \quad . \]
    Thus $\cos( \theta) \le \frac{3}{4}a$, as required.    
\end{proof}

\subsection{The configuration space as a trivial circle bundle} \label{sec: triv} 

Let $n\ge2$, and consider the free continuous $S^1$-action on $\Cf_{n,r}$ given by rotating a configuration about the origin, resulting from the rotational symmetry of $D^2$. Take the projection $q \colon \Cf_{n.r} \to \qCf_{n,r} \coloneqq \Cf_{n,r}/S_1$, which is a principal $S^1$-bundle, and choose any section $s\colon \qCf_{n,r} \to \Cf_{n,r}$. Let $\Theta \colon \Cf_{n,r} \to S^1$ map each $\DD\in \Cf_{n,r}$ to the unique $\theta \in S^1$ such that $\DD = \theta \cdot s[\DD]$. This yields the following.

\begin{prop} \label{prop: triv}
    $p \colon \Cf_{n,r} \to \qCf_{n,r} \times S^1, \ \DD \mapsto \left( q(\DD), \Theta (\DD) \right)$ is a homeomorphism with inverse $\left( \bar{\DD}, \theta \right) \mapsto \theta \cdot s\left( \bar{\DD} \right)$. Furthermore, $s$ is a homeomorphism onto its image.
\end{prop}

Here, we will use the section mapping each $\bar{\DD}\in \qCf_{n,r}$ to the unique representative in which $x_3-x_2$ points in the positive $x$-direction (i.e. $x_3-x_2 = \left(|x_3-x_2| ,0 \right)$ in Cartesian coordinates). Then $\Theta(\DD)$ is the anticlockwise angle from the vector $(1,0)$ to the vector $x_3-x_2$.

\subsection{The geometric structure of the quotient space} \label{sec: struc}

Prop. \ref{prop: triv} shows that $\qCf_{4,\frac{1}{3}}$ has the same homotopy-type as $s\left( \qCf_{4, \frac{1}{3}} \right) = \{\DD \in \Cf_{4,\frac{1}{3}} \mid \Theta (\DD) = 0 \}$.  Let $\bar{\tau} = \tau|_{s\left( \qCf_n \right)}$ for all $n\in \NN$. Then $s\left( \qCf_{4, \frac{1}{3}} \right) = \bar{\tau} ^{-1} \left[ \frac{1}{3}, \infty \right)$.

If $x\in \Cf_4$ has some point at the origin, then $\tau(x) \le \frac{1}{3}$. Therefore, we may begin by writing $\bar{ \tau}^{-1} \left[ \frac{1}{3}, \infty \right) = A\cup B\cup C$, where 
\begin{align*}
    A & = \left\{x \in \bar{ \tau}^{-1}\left[ \frac{1}{3}, \infty \right) \colon \ \forall_j \ |x_j| > 0 \right\}, \\
    B & = \left\{x \in \bar{ \tau}^{-1}\left( \frac{1}{3} \right) \colon \ \exists_i \ |x_i| = 0 \textrm{ and all } x_j \textrm{ lie in some closed semidisc} \right\} \textrm{, and} \\
    C & = \left\{x \in \bar{ \tau}^{-1}\left( \frac{1}{3} \right) \colon \ \exists_i \ |x_i| = 0 \textrm{ and no open semidisc contains all } x_j\ne 0 \right\}.
\end{align*}

Moreover, in $\bar{ \tau}^{-1} \left[ \frac{1}{3}, \infty \right)$, no two points can lie on the same radius of the unit disc unless one is at the origin; then in $A$, the four points have a well-defined cyclic ordering $\perm$ (taken anticlockwise) about the origin, so we may write $A = \bigsqcup_\perm A_\perm$.  Similarly, $B = \bigsqcup_{(i,\rho)} B_{i,\rho}$, where $i\in \{1,2,3,4\}$ is the index of the point at the origin and $\rho$ is an ordering of the remaining three points (this is not a cyclic ordering, since the empty open semi-disc gives a well-defined starting point), and $C = \bigsqcup_{(i,\rho')} C_{i,\rho'}$, where in this case $\rho'$ is a cyclic ordering of the remaining three points. 

We also define $S_\perm = A_\perm \cup \bigcup_{i\in \{1,2,3,4\}} B_{i, \hat{\perm}_i}$.

\begin{prop} \label{prop: structure}
    $S_\perm$ is the closure of $A_\perm$ in $\bar{ \tau}^{-1} \left[ \frac{1}{3}, \infty \right)$.
\end{prop}
\begin{proof}
    Let $(x^m)_{m\in \NN}$ be a sequence in $A_\perm$ converging to some $x \in \bar{ \tau}^{-1} \left[ \frac{1}{3}, \infty \right)$. In particular, $x^m_i \to x_i$ for each $i\in \{1,2,3,4\}$. If there is no $i$ such that $x_i = 0$, then  $x \in A_\perm$.
    
    Otherwise suppose $x_i = 0$ for some $i\in \{1,2,3,4\}$, and write $\perm = [i\ i_1\ i_2\ i_3]$. Take $\varepsilon>0$.  Then for all sufficiently large $m$, we have $|x^m_j - x_j| < \varepsilon$ for all $j\in \{1,2,3,4\}$. Therefore, we can choose some $\delta>0$, where $\delta\to 0$ as $\varepsilon \to0$, such that $|\angle x_{j} 0 x^m_{j}| < \delta$ for all $j\ne i$. Combining this with Cor. \ref{lem: angle at zero}, we have 
    \[\angle x^m_i 0 x_{i_1} \ge \angle x^m_i 0 x^m_{i_1} - |\angle x_{i_1} 0 x^m_{i_1}| \ge \arccos \frac{3}{4} |x^m_i| - \delta \]
    and similarly $\angle x_{i_3} 0 x^m_{i} \ge \arccos \frac{3}{4} |x^m_i| - \delta$.
    
    The radii through $x_{i_1}, x_{i_2}, x_{i_3}$ split the unit disc into three sectors. Since $|\angle x_{j} 0 x^m_{j}| < \delta$ for all $j\ne i$, the points $x_{i_1}, x_{i_2}, x_{i_3}$ retain the cyclic order $[i_1\ i_2\ i_3]$, and $x^m_i$ is in the sector between $x_{i_3}$ and $x_{i_1}$. This sector has angular extent \[\angle x_{i_3} 0 x^m_{i} + \angle x^m_i 0 x_{i_1} \ge 2 \left(\arccos \frac{3}{4} |x^m_i| - \delta \right) \to \pi \qquad \textrm{as } m\to \infty \quad .\] Therefore, $x\in B_{i, \hat{\perm}_i}$.

    Conversely, every $x\in A_\perm$ is the limit of a constant sequence in $A_\perm$, and every $x\in B_{i, \hat{\perm}_i}$ is the limit of the sequence $(x^m)_{m\in \NN}$ in $A_\perm$ given by $x^m_j = x_j$ for all $m$ and all $j\ne i$, and $x_i \to 0$ along the radius which bisects the largest sector not containing any other $x_j$.
\end{proof}

\begin{prop} \label{prop: structure2}
    $C_{i,\rho'}$ is parametrised by the hexagon \[\left\{ (\phi_1, \phi_2) \in \RR^2 \colon \frac{\pi}{3} \le \phi_1 \le \pi, \frac{\pi}{3} \le \phi_2 \le \pi, \pi \le \phi_1 + \phi_2 \le \frac{5\pi}{3} \right\} \quad .\]
    $S_\perm \cap C_{i, \rho'}$ is one of the edges $\phi_1 = \pi$, $\phi_2 = \pi$, $\phi_1 + \phi_2 = \pi$ if $\rho' = [\hat{\perm}_i]$; otherwise, it is empty.
\end{prop}

\begin{figure}[htbp]
    \begin{center}
        \begin{tikzpicture}
            \draw (0,0) circle [radius=3];
        \begin{scope}[scale=1.3, rotate = 139.2];
            \draw[color=black!50] (0,0) circle [radius=0.67];
            \draw[color=black!50] (1.33,0) circle [radius=0.67];
            \draw[color=black!50] (-0.2, -1.32) circle [radius=0.67];
            \draw[color=black!50] (-1.32, 0.2) circle [radius=0.67];
            \node[circle, inner sep = 0, minimum size = 4, fill=black, label = above:$x_2$] at (1.33,0) {};
            \node[circle, inner sep = 0, minimum size = 4, fill=black, label = right:$x_1$] at (0, 0) {};
            \node[circle, inner sep = 0, minimum size = 4, fill=black, label = above left:$x_3$] at (-0.2, -1.32) {};
            \node[circle, inner sep = 0, minimum size = 4, fill=black, label = above right:$x_4$] at (-1.32, 0.2) {};
            \draw (0,0) -- (2,0);
            \draw (0,0) -- (-0.3, -1.98);
            \draw (0,0) -- (-1.98, 0.3);
            \draw (0.5,0) arc (0:-98.6:0.5);
            \draw (0.4,0) arc (0:171.4:0.4);
            \node at (0,0.6) {$\phi_1$};
            \node at (0.5,-0.5) {$\phi_2$};
        \end{scope};
        \begin{scope}[scale=0.65, shift = {(7,-3.5)}];
            \draw[very thick, ->] (-1,0) -- (8,0);
            \draw[very thick, ->] (0,-1) -- (0,8);
            \node at (8.3,0) {$\phi_1$};
            \node at (0,8.3) {$\phi_2$};
            \draw[pattern=dots, pattern color=blue!60!black, draw=black] (2*pi,2*pi/3) -- (2*pi,4*pi/3) -- (4*pi/3,2*pi) -- (2*pi/3,2*pi) -- (2*pi/3,4*pi/3) --  (4*pi/3,2*pi/3) -- cycle;
            \draw[dashed, thin] (-1,2*pi/3) -- (4*pi/3,2*pi/3) -- (2*pi+1,-1); 
            \draw[dashed, thin] (2*pi/3,-1) -- (2*pi/3,4*pi/3) -- (-1, 2*pi+1);
            \draw[dashed, thin] (-1,2*pi) -- (2*pi/3,2*pi) -- (2*pi/3,8);
            \draw[dashed, thin] (2*pi,-1) -- (2*pi,2*pi/3) -- (8, 2*pi/3);
            \draw[dashed, thin] (2*pi,8) -- (2*pi,4*pi/3) -- (8, 10*pi/3 - 8);
            \draw[dashed, thin] (8, 2*pi) -- (4*pi/3, 2*pi) -- (10*pi/3-8, 8);
            \node at (2*pi/3-0.3, -0.4) {$\frac{\pi}{3}$};
            \node at (-0.3, 2*pi/3-0.4) {$\frac{\pi}{3}$};
            \node at (2*pi-0.3, -0.3) {$\pi$};
            \node at (-0.3, 2*pi-0.3) {$\pi$};
            \node at (9,3.7) {$\phi_1+\phi_2=\frac{5\pi}{3}$};
            \draw[red, line width = 0.7mm] (2*pi,2*pi/3) -- (2*pi,4*pi/3);
            \draw[red, line width = 0.7mm] (2*pi/3,2*pi) -- (4*pi/3,2*pi);
            \draw[red, line width = 0.7mm] (2*pi/3,4*pi/3) -- (4*pi/3,2*pi/3);
        \end{scope};
        \end{tikzpicture}
    \end{center}
    \caption{Left: A configuration in one of the connected components, $C_{1, [2\ 4\ 3]}$, of the critical locus of $\tau$ at $r= \frac{1}{3}$. Right: A coordinate parametrisation of $C_{1, [2\ 4\ 3]}$, using the angles shown on the configuration, given by the blue dotted area including boundary. The thick red line segments denote the intersection with the pieces $S_\perm$.}
    \label{fig: C-struct} 
\end{figure} 

\begin{proof}
    Let $\rho' = [j\ k\ l]$. Then let $\hat{x} = (1,0)$, $\phi_1 = \angle x_j 0 x_k$, $\phi_2 = \angle x_l 0 x_j$. A configuration in $\{x\in \tau^{-1} \left[ \frac{1}{3}, \infty \right) \mid \exists_i |x_i|=0 \textrm{ and no open semidisc contains all } x_j\ne 0\}$ is uniquely determined by the triple $(\angle \hat{x} 0 x_j, \phi_1, \phi_2)$. Furthermore, $\angle \hat{x} 0 x_j$ is uniquely determined by the pair $(\phi_1, \phi_2)$ and the condition $\Theta(x) = 0$. Thus, a configuration in $C_{i,\rho'}$ is uniquely determined by $(\phi_1, \phi_2)$.
    
    $\phi_1, \phi_2 \ge \frac{\pi}{3}$ and $\phi_1 + \phi_2 \le \frac{5\pi}{3}$ by Cor. \ref{lem: angle}. $\phi_1, \phi_2 \le \pi$ and $\phi_1 + \phi_2 \ge \pi$ by the condition that the remaining points are not contained in an open semidisc. When the three latter inequalities are all strict, then $x\notin A\cup B$. On the other hand, if $\phi_1 = \pi$, then $x_j, x_i, x_k$ lie on some common diameter, so $x\in B_{i, (k\ l\ j)} \subset S_{[i\ k\ l\ j]}$. Similarly, $x\in B_{i, (j\ k\ l)} \subset S_{[i\ j\ k\ l]}$ if $\phi_2=\pi$, and $x\in B_{i, (l\ j\ k)} \subset S_{[i\ l\ j\ k]}$ if $\phi_1 + \phi_2 = \pi$. 
\end{proof}

Thus, $\qCf_{4, \frac{1}{3}}$ consists of the six pieces $S_\perm$ and the eight hexagonal pieces $C_{i, \rho'}$. By the arguments of Prop. \ref{prop: structure2}, $S_\perm$ is connected to each of the four pieces $C_{i, \rho'}$ where $i\in \{1,2,3,4\}$ and $\rho'$ is the cyclic ordering $\hat{\perm_i}$ obtained by removing $i$ from $\perm$. Geometrically, $\hat{\perm_i}$ is the cyclic ordering of the remaining points when $x_i$ flows to the origin. Conversely, each $C_{i, \rho'}$ is connected to three pieces $S_\perm$, depending on which position you insert $i$ into $\rho'$; and geometrically, the insertion of $i$ between $j$ and $k$ inside $\rho'$ corresponds to moving the outer discs until the empty arc between $x_j$ and $x_k$ is at least $\pi$ radians, and then flowing $x_i$ orthogonally to $x_j-x_k$ into this empty region. Each intersection between some $S_\perm$ and some $C_{i, \rho'}$ is a submersion of an interval. This is depicted in Fig. \ref{fig: struct}.

\begin{figure}[htbp]
    \centering
    \includegraphics[width=0.35\linewidth]{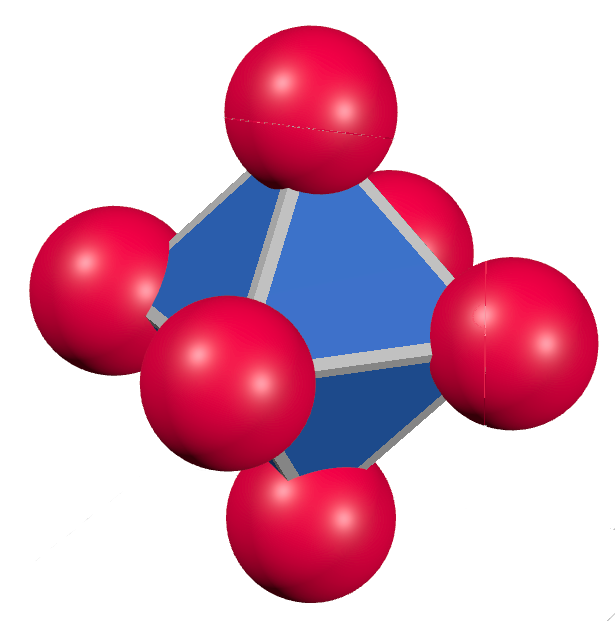}
    \begin{tikzpicture}[radius=0.9]
    	\begin{scope};
        \draw[dashed] (0,0) circle node {\footnotesize $[1\ 3\ 4\ 2]$};
        \draw (0,3) circle node{\footnotesize $[1\ 4\ 3\ 2]$};
        \draw (0,6) circle node {\footnotesize $[1\ 4\ 2\ 3]$};
        \draw[dash dot dot] (2.6,1.5) circle node {\footnotesize $[1\ 3\ 2\ 4]$};
        \draw (2.6,4.5) circle node {\footnotesize $[1\ 2\ 4\ 3]$};
        \draw (2.6,7.5) circle node {\footnotesize $[1\ 2\ 3\ 4]$};
        \draw (0,0.9) -- node[left] {$a$} (0,2.1);
        \draw (0,3.9) -- (0,5.1);
        \draw (0,6.9) -- node[left] {$d$} (0,8.1);
        \draw[dashed] (0,8.1) arc[start angle = -90, end angle = -30] node[inner sep = 0] (4) {};
        \draw (4) -- node[above] {$c$} (1.8,7.9);
        \draw (-0.8, 5.5) -- node[above] {$d$} (-1.8,5);
        \draw[dashed] (-1.8,5) arc[start angle = 30, end angle = -30] node[inner sep = 0] (3) {};
        \draw (3) -- node[below] {$a$} (-0.8, 3.5);
        \draw (2.6, 2.4) -- node[right] {$b$} (2.6, 3.6);
        \draw (2.6,5.4) -- (2.6, 6.6);
        \draw (0.8,0.5) -- node[below] {$e$} (1.8,1);
        \draw (0.8,3.5) -- (1.8,4);
        \draw (0.8,6.5) -- (1.8,7);
        \draw (0.8,2.5) -- (1.8,2);
        \draw (0.8,5.5) -- (1.8,5);
        \draw (3.4,7) -- (4.4,6.5);
        \draw (3.4,8) -- node[above] {$c$} (4.4, 8.5); 
        \draw[dashed] (4.4,8.5) arc[start angle = -150, end angle = -90] node[inner sep = 0] (1) {}; 
        \draw (1) -- node[right] {$e$} (5.2, 6.9);
        \draw[dash dot dot] (5.2, 6.9) arc[start angle = 90, end angle = 210] node[inner sep = 0] (2) {};
        \draw (2) -- node[below] {$b$} (3.4,5);   
        \node at (0.85, 1.5) {\footnotesize $4, [1\ 3\ 2]$};
        \node at (1.75, 3) {\footnotesize $1, [3\ 2\ 4]$};
        \node at (0.85, 4.5) {\footnotesize $2, [1\ 4\ 3]$};
        \node at (1.75, 6) {\footnotesize $4, [1\ 2\ 3]$};
        \node at (0.85, 7.5) {\footnotesize $1,[2\ 3\ 4]$};
        \node at (-0.85, 4.5) {\footnotesize $3, [1\ 4\ 2]$};
        \node at (3.45, 6) {\footnotesize $3, [1\ 2\ 4]$};
        \node at (4.35, 7.5) {\footnotesize $2, [1\ 3\ 4]$};
        \end{scope};
    \end{tikzpicture}
    \caption{Left: The structure of $\Cf_{4, \frac{1}{3}} /S^1$, shown on the surface of an octahedron (grey). On the faces are the connected components, $C_{i, \rho'}$ (blue), of the critical locus at $r=\frac{1}{3}$, represented as hexagons due to Prop. \ref{prop: structure2}. At the vertices are the connected components, $S_\perm$ (red), of the non-critical subset of $\Cf_{4, \frac{1}{3}} /S^1$.  Right: A net/schematic for $\Cf_{4, \frac{1}{3}} /S^1$. The circle containing the cyclic ordering $\perm$ represents the set $S_\perm$, and the hexagon containing a pair $i, \rho$ represents the critical component $C_{i,\rho}$. The straight outer edges are labelled with letters to show how they are glued. The boundary arcs are glued to the circles of the same line style.} 
    \label{fig: struct}
\end{figure}

\subsection{Contracting the configuration space} \label{sec: contr}

We note that $C$ is the closure of the locus of critical configurations $x\in s\left( \qCf_4 \right)$ with $\bar{\tau}(x)=\frac{1}{3}$ -- that is, there is no local flow at any $x \in C$ such that $\tau$ is increasing to first order. Conversely, if $x\in B$, then we may immediately flow the disc at the origin outwards as follows (see Fig. \ref{fig: out}). The rays from the origin through the three other disc centres divides the boundary of the unit disc into three arcs. Since all discs lie in a closed semidisc, precisely one of these arcs is at least a semicircle. Then the disc at the origin may move radially outwards towards the point bisecting this arc. We use this idea to retract $\hat{\tau}^{-1}\left[ \frac{1}{3}, \infty \right)$ onto the 1-skeleton $Y$ from Thm. \ref{thm: graph2} by first contracting each $S_\perm$ to a point and then retracting each $C_{i, \rho'}$ onto a 1-skeleton.

\begin{figure}[htbp]
    \begin{center}
    \begin{tikzpicture}[node distance={15mm}, thick, main/.style={}]
        \begin{scope}[rotate = 90];
            \draw[color=black] (0,0) circle [radius=2];
            \draw[color=black!50] (0,0) circle [radius=0.67];
            \draw[color=black!50] (0,1.33) circle [radius=0.67];
            \draw[color=black!50] (-0.2, -1.32) circle [radius=0.67];
            \draw[color=black!50] (-1.32, 0.2) circle [radius=0.67];
            \node[circle, inner sep = 0, minimum size = 4, fill=black, label = above:$x_2$] at (0, 1.33) {};
            \node[circle, inner sep = 0, minimum size = 4, fill=black, label = below:$x_3$] at (0, 0) {};
            \node[circle, inner sep = 0, minimum size = 4, fill=black, label = above:$x_1$] at (-0.2, -1.32) {};
            \node[circle, inner sep = 0, minimum size = 4, fill=black, label = above:$x_4$] at (-1.32, 0.2) {};
            \draw[->] (0.1, 0) -- (1.9,-0.05);
        \end{scope};        
    \end{tikzpicture}
    \end{center}
    \caption{An example of a non-critical configuration in $s \left( \qCf_{4, \frac{1}{3}} \right)$ with one disc centred at the origin. Since there is an empty sector of at least $\pi$ radians anticlockwise from $x_1$ to $x_2$, it is possible for disc 3 to move into this gap along some radius (e.g. the radius bisecting this sector), thus increasing the value of $\tau$ on the underlying configuration of points.}
    \label{fig: out}
\end{figure}

Let $s_\perm \in S_\perm$ be the unique configuration in which $|x_i|= \frac{2}{3}$ for all $i\in \{1,2,3,4\}$ and \[ \angle x_1 0 x_{\perm(1)} = \angle x_{\perm(1)} 0 x_{\perm^2(1)} = \angle x_{\perm^2(1)} 0 x_{\perm^3(1)} = \frac{\pi}{2}\] -- that is, if we consider $s_\perm$ as an element of $\Cf_{4, \frac{1}{3}}$, then it is the unique configuration in which all discs touch the boundary with cyclic order $\perm$, the disc centres lie on the corners of a square, and $x_3-x_2$ points in the positive $x$-direction.

\begin{prop} \label{prop: retr1}
    $\{s_\perm\}$ is a strong deformation retract of $S_\perm$.
\end{prop}

This claim is proven by the following three lemmas. Take $\varepsilon \in \left(0, \frac{1}{3} \right)$, and consider the sets 
\begin{align*}
    S_\perm' & = \left\{ x\in S_\perm \mid \forall_i \ |x_i|\ge \varepsilon \right\}, \\
    T_\perm &= \left\{ x\in S_\perm \mid \forall_i \ |x_i|= \frac{2}{3} \right\}.
\end{align*}

\begin{lem} \label{lem: retr1}
    $S_\perm'$ is a strong deformation retract of  $S_\perm$.
\end{lem}
\begin{proof}
   We start by defining our retraction. For any $x \in \Cf_{4, \frac{1}{3}}$, there is at most one $i$ such that $|x_i|\le \varepsilon$ by the triangle inequality. In the case that such $i$ exists, the radii through the remaining three points split the unit disc into three sectors. If $x_i=0$, one arc is at least a semicircle. Otherwise, the sector containing $x_i$ has extent at least $2\arccos \frac{3}{4}|x_i|$ by Cor. \ref{lem: angle at zero}, while the other two sectors each have extent at least $\frac{\pi}{3}$ by Cor. \ref{lem: angle}, and therefore their extent is at most $\frac{5\pi}{3} - 2\arccos \frac{3}{4}|x_i|$. Given that $|x_i| \le \frac{1}{3}$, we have \[2\arccos \frac{3}{4} |x_i| > 2\arccos \frac{1}{4} > \frac{5\pi}{3} - 2\arccos \frac{1}{4} > \frac{5\pi}{3} - 2\arccos \frac{3}{4} |x_i| \quad .\] Therefore there is a unique largest sector, so we may define the following: let $v$ be the point which bisects the arc bounding the largest sector, and $l$ be the unique positive real number such that $|x_i + lv| = \varepsilon$ (where this equation treats $v$ as a vector with $|v|=1$). Then we can define a continuous map $\chi \colon [0,1] \times S_\perm \to \Cf_4$ (shown in Fig. \ref{fig: small flow}), given by
    \[ \chi(t,x) \coloneqq \chi_t(x_1, x_2, x_3, x_4) = 
    \left( \left\{ \begin{array}{cl}
        x_j + tlv & \colon |x_j| \le \varepsilon \\
        x_j & \colon \textrm{otherwise}
    \end{array} \right. \right) _{j\in \{1,2,3,4\} } \quad .\]
    $\chi_t$ also preserves the cyclic order $\perm$, since $x_i$ flows in the direction of the angle bisector of the sector it lies in. We want to find a flow on $S_\perm$, that is, a continuous map $[0,1] \times S_\perm \to S_\perm$. To do this, we will first show that $\chi\left( [0,1] \times S_\perm \right) \subset \tau^{-1}\left[ \frac{1}{3}, \infty \right)$, and then modify $\chi$ so that $\Theta(\chi_t(x)) = 0$ for all $t\in [0,1]$, $x\in S_\perm$.
    
    We claim that $\chi\left( [0,1] \times S_\perm \right) \subset \tau^{-1}\left[ \frac{1}{3}, \infty \right)$. $\chi_t(x_j)$ is stationary for any $|x_j| \ge \varepsilon$, so we only need to check the case that $|x_i| \le \varepsilon$ for some $i$, and moreover, in this case, $|\chi_t(x_i)| \le \varepsilon < \frac{2}{3}$ for all $t\in [0,1]$. Then it remains to show that $|\chi_t(x_i) - x_j| \ge \frac{2}{3}$ for all $j\ne i$. Now suppose by symmetry that $x_i$ is anticlockwise of $v$ inside the sector bounded by $x_{\perm^3(i)}$ and $x_{\perm(i)}$. We have
    \begin{align*}
        |\chi_t(x_i) - x_j|^2 & = |x_i - x_j + tlv|^2 \\
        & = |x_i - x_j|^2 + 2tl(x_i - x_j) \cdot v + t^2l^2|v|^2 \\
        & = |x_i - x_j|^2 + 2tl \left(|x_i| \cos (\angle v0x_i) - |x_j| \cos (\angle v0x_j) \right)  + t^2l^2
    \end{align*}
    We bound this below by the following inequalities. By Cor. \ref{lem: angle}, we have 
    \begin{equation}
        \angle x_{\perm(i)} 0 x_{\perm^2(i)} \ge \frac{\pi}{3}, \quad \angle x_{\perm^2(i)} 0 x_{\perm^3(i)} \ge \frac{\pi}{3}
    \end{equation}
    hence 
    \begin{equation}
        \angle v0 x_{\perm(i)} = \frac{1}{2} \left( 2\pi - \angle x_{\perm(i)} 0 x_{\perm^2(i)} - \angle x_{\perm^2(i)} 0 x_{\perm^3(i)} \right) \le \frac{1}{2} \left(2\pi - 2\frac{\pi}{3} \right) = \frac{2\pi}{3}
    \end{equation}
    By Cor. \ref{lem: angle at zero}, we have
    \begin{equation}
        \angle x_i0 x_{\perm(i)} \ge \arccos \frac{3}{4}|x_i| \ge \arccos \frac{1}{4} \ge \frac{5\pi}{12}
    \end{equation}
    and therefore 
    \begin{equation}
        \angle x_{\perm^3(i)} 0v = \angle v0x_{\perm(i)} = \angle v0x_i + \angle x_i0 x_{\perm(i)} \ge \angle v0x_i + \arccos \frac{3}{4}|x_i| \ge \angle v0x_i + \frac{5\pi}{12}
    \end{equation}
    hence 
    \begin{equation}
        \angle v0x_i \le \frac{2\pi}{3} - \frac{5\pi}{12} = \frac{\pi}{4}   
    \end{equation}

    For $j=\perm(i)$, we have by (4.4) and (4.2)
    \[0 < \arccos \frac{3}{4}|x_i| \le \angle v0 x_{\perm(i)} \le \pi\]
    Therefore, since $\cos$ is decreasing on $[0,\pi]$ and $|x_{\perm(i)}| \le \frac{2}{3}$, and applying (4.5), we get
    \[|x_{\perm(i)}| \cos (\angle v0 x_{\perm(i)}) \le \frac{2}{3} \cos \left( \arccos \frac{3}{4}|x_i| \right) = \frac{1}{2}|x_i| < |x_i| \cos( \angle v0x_i)\]
    
    For $j=\perm^2(i)$, we have by (4.1) and (4.3)
    \[ \angle v0 x_{\perm^2(i)} = \angle v0x_i + \angle x_i 0 x_{\perm(i)} + \angle x_{\perm(i)} 0 x_{\perm^2(i)} \ge 0 + \frac{5\pi}{12} + \frac{\pi}{3} > \frac{\pi}{2} \]
    and by (4.1) and (4.4)
    \[ \angle v0 x_{\perm^2(i)} = 2\pi - \angle x_{\perm^2(i)} 0v =2\pi - \angle x_{\perm^2(i)} 0 x_{\perm^3(i)} - \angle x_{\perm^3(i)} 0v  \le 2\pi - \frac{\pi}{3} - \frac{5\pi}{12} < \frac{3\pi}{2} \]
    Therefore, \[|x_{\perm^2(i)}| \cos (\angle x_{\perm^2(i)} 0v) <0 <|x_i| \cos( \angle v0x_i)\]
    
    For $j=\perm^3(i)$, we have $\cos \angle v0 x_{\perm^3(i)} = \cos\angle x_{\perm^3(i)} 0v$ and $\angle x_{\perm^3(i)} 0v = \angle v0 x_{\perm(i)}$, so $|x_{\perm^3(i)}| \cos (\angle v0 x_{\perm^3(i)}) < |x_i| \cos( \angle v0x_i)$ by the same argument as the case $j=\perm(i)$.

    Thus the map $t\mapsto |\chi_t(x_i)-x_j|$ is increasing with respect to $t$ for all $j\ne i$, which completes our claim that $\tau(\chi_t(x)) \ge \frac{1}{3}$ for all $t\in [0,1]$, $x\in S_\perm$.
    
    The map $t\mapsto \Theta (\chi_t(x))$ is not necessarily constant, so $\chi_t$ may not preserve $s\left( \qCf_{4, \frac{1}{3}} \right)$. Thus we define a modified flow $\bar{\chi}_t(x) = \left( -\Theta(\chi_t(x))\right) \cdot \chi_t(x)$. This flow retains all the previous properties due to invariance of distance and angles under rotation, and moreover $\Theta( \bar{\chi}_t(x)) = -\Theta(\chi_t(x)) + \Theta(\chi_t(x)) = 0$. Thus $\bar{\chi}_t(x) \in S_\perm$ for all $x \in S_\perm$. It can be verified that $\bar{\chi}_t|_{S_\perm'}  = \mathrm{id}_{S_\perm'}$ for all $t\in [0,1]$ and $\bar{\chi}_1(S_\perm) = S_\perm'$. Therefore $(\chi_t)_{t\in [0,1]}$ is a strong deformation retraction of $S_\perm$ onto $S_\perm'$.    
\end{proof}

\begin{lem} \label{lem: retr2}
    $T_\perm$ is a strong deformation retract of $S_\perm'$.
\end{lem}
\begin{proof}
   Let $V = \left\{ x \in \Cf_4 \mid \forall_{\lambda \ge 0} \forall_{i\neq j} x_i \neq \lambda x_j\right\} \supset S_\perm'$, so that $V$ consists of all configurations in $\Cf_4$ in which no point is at the origin and no two points lie on the same radius of the unit disc. $V' \coloneqq \left\{x \in \Cf_4 \colon \ \forall_i \ |x_i| = \frac{2}{3} \right\} \supset T_\perm$ can be shown to be a strong deformation retract of $V$ by the following flow, defined for $t\in [0,1]$. 
    \[\phi_t(x_1, x_2, x_3, x_4) = \left( \left( 1-t+\frac{2t}{3|x_i|} \right) x_i \right)_{i \in \{1,2,3,4\}} \]
    Here, each point of the configuration flows along the radius it lies on at constant velocity, until it has magnitude $\frac{2}{3}$ (and hence the cyclic order of the points is unchanged). 
    
    $\Theta(\phi_t(x))$ may not be constant with respect to $t$, so we define a modified flow $\bar{\phi}_t(x) = (-\Theta( \phi_t(x))) \cdot \phi_t(x)$,
    Then $\Theta(\bar{\phi}_t(x)) = 0$ for all $x \in V$, $t\in [0,1]$, so in particular, $\bar{\phi}_t$ preserves the section $s\left( \qCf_{4, \frac{1}{3}} \right)$. Moreover, for all $t\in[0,1]$ and $x\in S_\perm'$, and all distinct $i,j \in \{1,2,3,4\}$, we have $\varepsilon \le |x_i| \le |\bar{\phi}_tx_i| \le \frac{2}{3}$ and $|x_i-x_j| \ge \frac{2}{3}$; therefore by Lemma \ref{lem: flow dist}, $|\bar{\phi}_tx_i- \bar{\phi}_tx_j| \ge \frac{2}{3}$ -- that is, $\bar{\phi}_tx\in S_\perm'$. Therefore $(\bar{ \phi}_t|_{S_\perm'})_{t\in [0,1]}$ is a strong deformation retraction from $S_\perm'$ onto $V' \cap S_\perm' = T_\perm$.
\end{proof}

\begin{lem} \label{lem: retr3}
    $\{s_\perm\}$ is a strong deformation retract of $T_\perm$.
\end{lem}
\begin{proof}
    Let $\tilde{T}_\perm = \{ x\in \tau^{-1} \left[ \frac{1}{3}, \infty \right) \mid \forall_i |x_i|= \frac{2}{3}, \textrm{ the points have cyclic order } \perm\} \supset T_\perm$. For any $x\in \tilde{T}_\perm$, let $\alpha = \angle x_10 x_{\perm (1)}$, $\beta = \angle x_10 x_{\perm ^2(1)}$, $\gamma = \angle x_10 x_{\perm ^3(1)}$. Denote by $R_\theta$ the $S^1$-action on the plane given by anticlockwise rotation about the origin through angle $\theta$, and consider the following flow on $\tilde{T}_\perm$, defined for $t\in [0,1]$. \[ \psi_t(x_1, x_{\perm(1)}, x_{\perm^2(1)}, x_{\perm^3(1)}) = \left( x_1, R_{t\left( \frac{\pi}{2} -\alpha \right)} x_{\perm(1)}, R_{t\left( \pi -\beta \right)} x_{\perm^2(1)}, R_{t\left( \frac{3\pi}{2} -\gamma \right)} x_{\perm^3(1)} \right) \] $\psi_t$ fixes $x_1$, while causing each of the remaining $x_i$ to flow along an arc of radius $\frac{2}{3}$ centred on the origin until it lies on the correct corner of the square which has $x_1$ as a vertex. 
    
    We want to prove that $\psi_t$ is well-defined, that is, that $\psi_t(x) \in \tilde{T}_\perm$ for all $t\in [0,1], \ x\in \tilde{T}_\perm$. Given that $|\psi_t(x_i)| = \frac{2}{3}$ for all $t\in[0,1]$ and all $i$, it remains to check $|\psi_t (x_i)- \psi_t( x_j)| \ge \frac{2}{3}$ for all $i\ne j$, or equivalently $\angle \psi_t( x_i) 0 \psi_t( x_{\perm(i)}) \ge \frac{\pi}{3}$ for all $i$ and all $t\in[0,1]$ (this condition also ensures that the cyclic order of the points is unchanged). This is immediate for $i=1$, since either $\alpha>\frac{\pi}{2}$, in which case $\angle \psi_t(x_1) 0 \psi_t(x_{\perm(1)}) = \angle x_1 0 R_{t(\pi/2-\alpha)}( x_{\perm(1)}) \ge \frac{\pi}{2}$ for all $t\in [0,1]$, or else $\angle \psi_t(x_1) 0 \psi_t(x_{\perm(1)})$ is increasing with respect to $t$. But by considering instead the flow $(t,x) \mapsto (t(\alpha -\frac{\pi}{2})) \cdot \psi_t(x)$, we obtain a rotation of our original flow, where now $x_{\perm(1)}$ is stationary; since distances are invariant under rotation, we may argue in the same way that the desired property holds also for $i=\perm(1)$, and similarly for the remaining values of $i$. 
    
    $\Theta( \psi_t(x))$ is not necessarily constant with respect to $t$, so we instead use the modified flow $\bar{ \psi}_t(x) = (-\Theta(\psi_t(x))) \cdot \psi_t(x)$. Then $\Theta( \bar{ \psi}_t (x)) = 0$ for all $t\in [0,1]$, so in particular $\bar{\psi}_t$ preserves the section $s\left( \qCf_{4, \frac{1}{3}} \right)$, where $T_\perm = \tilde{T}_\perm \cap s\left( \qCf_{4, \frac{1}{3}} \right)$. Given that $\bar{\psi}_t(s_\perm) = s_\perm$ for all $t\in [0,1]$ and $\bar{\psi}_1(x) = s_\perm$ for all $x\in T_\perm$, it follows that $(\bar{ \psi}_t|_{T_\perm})_{t\in [0,1]}$ is a strong deformation retraction from $T_\perm$ onto $\{s_\perm \}$.
\end{proof}

\begin{figure}[htbp]
    \begin{center}
    \begin{tikzpicture}[node distance={15mm}, thick, main/.style={}]
    	\begin{scope}[scale=0.6];
            \draw[color=black] (0,0) circle [radius=5] node[] {$\times$};
            \draw[dotted] (0,0) circle [radius=1.6];
            \node at (1.4, -1.4) {$B(0, \varepsilon)$};
            \draw[->, thin] (0,0) -- (5,0) node[right] {$v$};
            \node[circle, inner sep = 0, minimum size = 4, fill=black, label = above:$x_{\perm(i)}$] at (-0.83, 3.09) {};
            \node[circle, inner sep = 0, minimum size = 4, fill=black, label = above left:$x_i$] at (0.7,0.5) {};
            \draw[->] (0.75, 0.5) -- node[above right] {$\chi_t(x_i)$} (1.5, 0.5) ;
            \node[circle, inner sep = 0, minimum size = 4, fill=black, label = above:$x_{\perm^2(i)}$] at (-2.90, 0.65) {};
            \node[circle, inner sep = 0, minimum size = 4, fill=black, label = above:$x_{\perm^3(i)}$] at (-0.78, -2.90) {};
            \end{scope};
    \end{tikzpicture}
    \end{center}
    \caption{The flow $\chi_t$ on $\Cf_{4, \frac{1}{3}}$ from Lemma \ref{lem: retr1}. $v$ bisects the sector containing $x_i$, bounded by $x_{\perm(i)}$ and $x_{\perm^3(i)}$, where $x_i$ is the unique point of $x$ with $|x_i|\le \varepsilon$. $x_i$ flows in the $v$-direction.}
    \label{fig: small flow}
    \end{figure}

The following proposition is the key ingredient in Thm. \ref{thm: graph2}, and is demonstrated visually in Fig. \ref{fig: shrink}.

\begin{figure}[htbp]
    \centering
    \includegraphics[width=0.3\linewidth]{Thesis_graphics/qConf_structure.png}
    \includegraphics[width=0.3\linewidth]{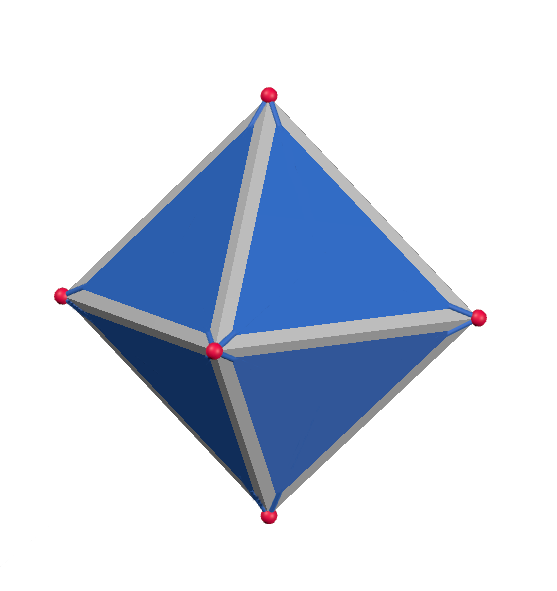}
    \includegraphics[width=0.3\linewidth]{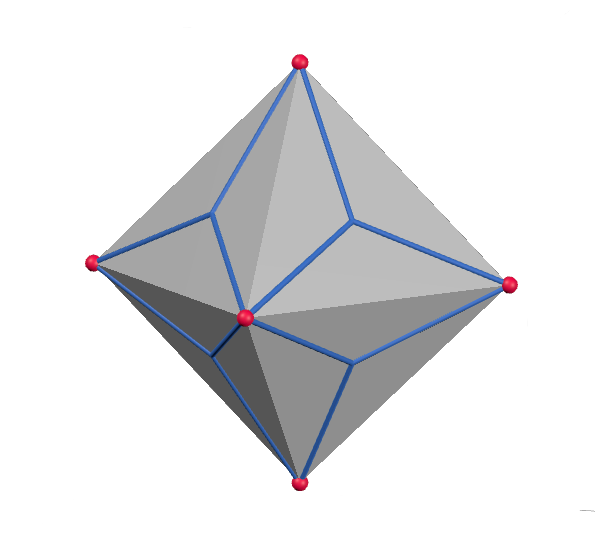}
    \caption{A graphical representation of Prop. \ref{prop: graph}. The left-hand picture, repeated from Fig. \ref{fig: struct}, shows the structure of $\Cf_{4, \frac{1}{3}} / S^1$. We contract the non-critical components (red) to points (shown in the middle picture), and then retract each critical component (blue) onto a Y-shaped skeleton (right).}
    \label{fig: shrink}
\end{figure}

\begin{prop} \label{prop: graph}
    Let $Y$ be the cell complex containing:
    \begin{itemize}
        \item a set, $V_1$, of 0-cells, indexed over the 6 vertices of an octahedron
        \item a set, $V_2$, of 0-cells, indexed over the 8 faces of an octahedron
        \item a set, $E$, of 1-cells, in which $v_1 \in V_1$ and $v_2 \in V_2$ bound a 1-cell if and only if the vertex corresponding to $v_1$ is on the boundary of the face corresponding to $v_2$.
    \end{itemize} 
    Then $\qCf_{4, \frac{1}{3}} \simeq Y$. 
\end{prop}

Since each face is bounded by three vertices, $|E|=8\times 3= 24$. The theorem will require a theorem of Borsuk \cite[Thm. 1]{bors}, and we must first demonstrate some properties of $s\left( \qCf_{n,r} \right)$.

\begin{lem} \label{lem: metsp}
    $s\left( \qCf_{n,r} \right)$ is a metric space.
\end{lem}
\begin{proof}
    First, note that $D^2$ is a metric space under the Euclidean metric. Thus, $\left( D^2 \right) ^n$ is a metric space. Since $s\left( \qCf_{n,r} \right)$ is a subspace of $\left( D^2 \right) ^n$, the result follows.
\end{proof}

\begin{lem} \label{lem: ANR}
    $s\left( \qCf_{n,r} \right)$ is an absolute neighbourhood retract.
\end{lem}
\begin{proof}
    First, note that we may describe this space as 
    \begin{IEEEeqnarray*}{rCr}
        s\left( \qCf_{n,r} \right) &=& \left\{ (x_1, \ldots, x_n) \in \left( \RR^2 \right) ^n \mid \forall_i \ |x_i| \le 1-r, \ \forall_{i\ne j} |x_i-x_j| \ge 2r, \right. \\
        && \left. \left((x_3-x_2) \cdot (1,0)\right)^2 = |x_3-x_2|^2, (x_3-x_2) \cdot (1,0) > 0 \right\}
    \end{IEEEeqnarray*}
    Given that it is described by a finite set of polynomial equations and inequalities under the identification $\left( \RR^2 \right)^n = \RR^{2n}$, $s\left( \qCf_{n,r} \right)$ is a semi-algebraic (and thus semi-analytic) set as defined in \cite[\S1]{loj}. Thus there is a homeomorphism between $s\left( \qCf_{n,r} \right)$ and a locally finite simplicial complex $L$ \cite[Thm. 1.c]{loj}. Therefore, in particular, $s\left( \qCf_{n,r} \right)$ is $\varepsilon$-dominated by $L$ for any $\varepsilon>0$ as defined in \cite[\S6]{hann}, and $L$ is an absolute neighbourhood retract \cite[Cor. 3.5]{hann}. Thus $s\left( \qCf_{n,r} \right)$ is an absolute neighbourhood retract \cite[Thm. 7.2.b]{hann}. 
\end{proof}

\begin{proof}[Proof of Prop. \ref{prop: graph}]
    First, we have that $S_\perm$ is closed in $s\left( \qCf_{n,r} \right)$ by Prop. \ref{prop: structure}. Furthermore, $s\left( \qCf_{n,r} \right)$ is a metric space by Lemma \ref{lem: metsp} and an absolute neighbourhood retract by Lemma \ref{lem: ANR}. Therefore $\left( s\left( \qCf_{n,r} \right), S_\perm \right)$ has the homotopy extension property with respect to $s\left( \qCf_{n,r} \right)$ by Borsuk's condition \cite[Thm. 1]{bors}.

    Now let $\sim$ be an equivalence on $s\left( \qCf_{n,r} \right)$, given by $x\sim y$ if and only if there is some cyclic ordering $\perm$ such that $x,y\in S_\perm$. Since $S_\perm$ is contractible by Prop. \ref{prop: retr1} and $\left( s\left( \qCf_{n,r} \right), S_\perm \right)$ has HEP, we may contract each $S_\perm$ to a point inside $s\left( \qCf_{n,r} \right)$ continuously by \cite[Prop 0.17]{hatch}. Therefore $s\left( \qCf_{4, \frac{1}{3}}\right) \simeq s\left( \qCf_{4, \frac{1}{3}} \right) / \sim$. The points $S_\perm$ compose the set $V_1$.
    Since the contractions of $S_\perm$ to a point also contract alternate edges of each critical component, $C_{i, \rho'}$, these become triangles, connected to the points of $V_1$ at their vertices. Thus $s\left( \qCf_{4, \frac{1}{3}} \right) / \sim$ consists of eight triangular faces, where each face is joined to three others at each of its vertices --it is an octahedron with the edges (but not the vertices) cut out. 
    
    Now, choose a point in the interior of each $C_{i, \rho'}$, and connect each of these points to the three vertices of its triangle by a line segment. These points compose $V_2$, and the line segments compose $E$. Now, we have chosen a Y-shape inside each $C_{i,\rho'}$, consisting of a 0-cell from $V_2$ at the centre, three 1-cells from $E$, and bounded by three 0-cells from $V_1$. We may retract every $C_{i,\rho'}$ onto this simultaneously. Therefore, by Prop. \ref{prop: triv}, \[\qCf_{4, \frac{1}{3}} \cong s\left( \qCf_{4, \frac{1}{3}}\right) \simeq s\left( \qCf_{4, \frac{1}{3}} \right) / \sim \; \, \simeq Y\]
\end{proof}

\begin{proof}[Proof of Thm: \ref{thm: graph2}]
    $\Cf_{4, r} \simeq \Cf_{4, \frac{1}{3}}$ by Fact \ref{fact: crits} and Prop. \ref{prop: crit4}. $\Cf_{4, \frac{1}{3}} \cong \qCf_{4, \frac{1}{3}} \times S^1$ by Prop. \ref{prop: triv}. Finally, $\qCf_{4, \frac{1}{3}} \times S^1 \simeq Y\times S^1$ by Prop. \ref{prop: graph}.
\end{proof}

\section{Persistence under deformation of the unit disc} \label{chap: deform}

Given the rotational symmetry of $D^2$, $\tau \colon \Cf_n \to \RR$ behaves as a Morse-Bott function (a generalisation of a Morse function in which critical points may not be isolated, defined in \cite[\S2.5]{bott}), with each critical radius corresponding not to a single critical configuration, but to an $S^1$ of critical configurations. Under a well-chosen small perturbation, we would thus expect, generically, a critical level $r_0$ to split into two critical radii, say $r_-<r_+$ near $r_0$. This is shown in the case of the first critical radius $\frac{1}{n}$ in the following proposition and lemma. We recall the stress graph $N(x)$ from Def. \ref{defn: stress}.

\begin{lem}\emph{(Adapted from \cite[\S2.5]{ell})} \label{lem: line} 
    Let $E\subset \RR^2$ be an ellipse of semi-major radius $a$ and eccentricity $e$, centred at the origin, whose major axis lies along the $x$-axis, and $p=(p_1,0)$ a point on the major axis of $E$. If $|p_1|>ae^2$, then $\dist(p,\partial E) = a-|p_1|$, and this is achieved at either $(-a,0)$ or $(a,0)$; otherwise, $\dist(p,\partial E) = \sqrt{\left(a^2-\frac{p_1^2}{e^2}\right) (1-e^2)}$, and this is achieved at $\left(\frac{p_1}{e^2}, \pm \sqrt{ \left( a^2 - \frac{p_1^2}{e^4} \right) \left( 1-e^2 \right)} \right)$.
\end{lem}
\begin{proof}
    $\partial E = \left\{ (x,y)\in \RR \mid \frac{x^2}{a^2}+\frac{y^2}{b^2}=1 \right\}$, where $b=a\sqrt{1-e^2}$. Since $\partial E$ is compact, there is some $(x',y')\in \partial E$ such that $\dist(p,\partial E) = |p-(x',y')| = \sqrt{(x'-p_1)^2+(y')^2}$. Furthermore, we know that $p$ lies on the normal line to $\partial E$ at $(x',y')$, which has parametric form $x=x'(1+\frac{t}{a^2})$, $y=y'(1+\frac{t}{b^2})$, $t\in\RR$. Then in particular, at $p$, we have $0=y'(1+\frac{t}{b^2})$; that is, either $y'=0$ or $t=-b^2$. 
    
    In the first case, we have $(x',y')=(\pm a,0)$, hence $|(x',y')-p|=a-|p_1|$. 
    
    In the second case, we have $p_1=x'\left( 1-\frac{b^2}{a^2} \right) = x'e^2$, that is $x' = \frac{p_1}{e^2}$, and therefore $1 = \frac{p_1^2}{a^2e^4}+\frac{(y')^2}{b^2}$, or equivalently 
    \begin{align*}
        (y')^2 & = b^2 \left( 1-\frac{p_1^2}{a^2e^4} \right) \\
        & = b^2 - \frac{p_1^2}{e^4} + \frac{p_1^2}{e^4} - \frac{b^2p_1^2}{a^2e^4} \\
        & = b^2 - \frac{p_1^2}{e^4} + \frac{p_1^2}{e^2} \\
        & = a^2(1-e^2) - \frac{p_1^2}{e^4} \left( 1-e^2 \right) \\
        & = \left( a^2 - \frac{p_1^2}{e^4} \right) \left( 1-e^2 \right) \quad .
    \end{align*}
    Then 
    \begin{align*}
        |(x',y')-p|^2 & = (x'-p_1)^2+(y')^2 \\
        & = p_1^2\left(\frac{1}{e^2}-1\right)^2 + \left( a^2 - \frac{p_1^2}{e^4} \right) \left( 1-e^2 \right) \\
        & = \frac{p_1^2}{e^4} - \frac{2p_1^2}{e^2} + p_1^2 + a^2(1-e^2) - \frac{p_1^2}{e^4} + \frac{p_1^2}{e^2} \\
        & = - \frac{p_1^2}{e^2} + p_1^2 + a^2(1-e^2) \\
        & = \left(a^2-\frac{p_1^2}{e^2}\right) (1-e^2) \quad .
    \end{align*}
    Since $p\in E$, we have $|x'|<a$, and therefore this case only applies in the case $|p_1|<ae^2$. In this case, we have 
    \begin{align*}
        \left(a^2-\frac{p_1^2}{e^2}\right) (1-e^2) & = a^2-a^2e^2-\frac{p_1^2}{e^2}+p_1^2 \\
        & = a^2-2ap_1+p_1^2 - \left( a^2e^2 - 2ap_1+\frac{p_1^2}{e^2} \right) \\
        & = (a-p_1)^2 - \left( ae-\frac{p_1}{e} \right)^2 \\
        & \le (a-p_1)^2 \quad ,
    \end{align*}
    so $\dist(p,\partial E)= \sqrt{\left(a^2-\frac{p_1^2}{e^2}\right) (1-e^2)}$. Otherwise $\dist(p,\partial E) = a-|p_1|$.
\end{proof}

\begin{prop}
    Take $n\in \NN$, and let $C$ denote the set of critical configurations of $\Cf_n$ with radius $\frac{1}{n}$ where the points lie in number order along the diameter. Consider the family of ellipses $E_e$ with semi-major radius 1 indexed by eccentricity $e\in \left[ 0,\sqrt{\frac{n-1}{n}} \right)$. There are exactly two continuous families, $\left(x^{e,+} \right)_e$ and $\left(x^{e,-} \right)_e$, of configurations such that 
    \begin{itemize}
        \item $x^{e,+}, x^{e,-}$ are critical configurations of $\Cf_n(E_e)$,
        \item $x^{0,+}, x^{0,-} \in C$.
    \end{itemize}  
\end{prop}
\begin{proof}
Let $\left(x^e \right)_e$ be a family of configurations such that $x^e$ is a critical configuration of $E_e$ and $x^0 \in C$. We choose coordinates such that $E_e$ is defined by the equation $x^2 + \frac{y^2}{1-e^2} = 1$.

First, let $e\in \left( 0, \sqrt{\frac{n-1}{n}} \right)$ be such that $N(x^e)\cong N(x^0)$. Then we deduce by Lemma \ref{lem: hull}\emph{i} that the points of $x$ lie along the straight line between the boundary vertices of $N(x^e)$. Furthermore, $N(x^e)$ is a normal to the boundary at both points of intersection. The normal at a point $(p,q)$ is $2\left(p, \frac{q}{1-e^2} \right)$. Since the normal at the other intersection point must be parallel to this, the intersection point must be $(-p, -q)$. Then it follows that $(p,q) - (-p,-q) = \lambda \left(p, \frac{q}{1-e^2} \right)$ for some $\lambda\in \RR$, which has solutions only if $p=0$ or $q=0$. Therefore $N(x^e)$ is either the major axis or the minor axis of $E_e$. We denote the configurations of this form $x^{e,+}$ and $x^{e,-}$ respectively. The points of these configurations are defined by $x^{e,+}_i = \left( \frac{2i-n-1}{n}, 0 \right)$ and $x^{e,-}_i = \left( 0, \frac{2i-n-1}{n}\sqrt{1-e^2} \right)$. 

Now suppose for a contradiction there is $e'\in \left[ 0, \sqrt{\frac{n-1}{n}} \right)$ such that $N(x^{e'}) \not \cong N(x^0)$, and let $e^* = \inf \left\{ e\in \left[ 0,\sqrt{\frac{n-1}{n}} \right) \mid N(x^e) \not \cong N(x^0) \right\}$. $N(x^{e'})$ cannot be a strict subgraph of $N(x^0)$, since any strict subgraph of $N(x^0)$ has an interior vertex of degree 1 on every component, and thus cannot be balanced by Lemma \ref{lem: hull}\emph{ii}. 

Let $f \colon \Cf_n \to \RR$ be any distance function from the definition of $\tau$ -- that is, $f(x)$ takes the form $\frac{1}{2}|x_i-x_j|$ for some distinct $i,j$, or $\dist(x_i,\partial D^2)$ for some $i$. Let $g\colon \left[ 0, \sqrt{\frac{n-1}{n}} \right) \to \RR, \ e \to \frac{f(x^{e})}{\tau(x^{e})}$. Then $g$ is continuous, and the edge corresponding to $f$ is found in $N(x^e)$ if and only if $g(e)=1$. Since $g$ is constant on $[0,e^*)$, it follows by continuity that $g(e^*) = g(0)$ for all choices of $f$, and thus $N(x^{e^*}) \cong N(x^{0})$, which lies along either the major axis or the minor axis of $E$.

Hence by assumption, there exists some edge which lies in $N(x^{e'})$ but not in $N(x^{e^*})$. There are two possibilities: either the points $x^{e'}_1$ or $x^{e'}_n$ are adjacent to more than one boundary vertex, or there is some choice of $f$ such that $g(e^*) > g(e') = 1$. Denote $B_e = B\left(0, \sqrt{1-e^2} \right)$. First, if $x^{e'}$ lies along the major axis, we have $|x^{e'}_1| = |x^{e'}_n| =\frac{n-1}{n} > (e')^2$, thus $\dist(x^{e'}_1, \partial E_{e'})$ and $\dist(x^{e'}_n, \partial E_{e'})$ are each achieved at a unique point by Lemma \ref{lem: line}. On the other hand, if $x^{e'}$ lies along the minor axis, we see that $x^{e'}_1, x^{e'}_n \in B_{e'} \subset E_{e'}$, where $\partial B_{e'}$ intersects $\partial E_{e'}$ only on the line containing $\bigcup x^{e'}$. Now $\dist \left( x^{e'}_1, \partial B_{e'} \right)$ is achieved only at the point of intersection closer to $x_1$; therefore this is also the unique point where $\dist(x^{e'}_1, \partial E_{e'})$ is achieved. The same argument holds for $\dist(x^{e'}_n, \partial E_{e'})$. 

Now only one possibility remains: there is some choice of $f$ such that $g(e^*) > g(e') = 1$. We consider each distance functions $f$ for which $g(e^*)>1$:
\begin{itemize}
    \item For $j-i\ge 2$, $\frac{1}{2}|x^{e^*}_j - x^{e^*}_i| = (j-i)\tau \left( x^{e^*} \right) \ge 2\tau \left( x^{e^*} \right)$.
    \item For $x^{e^*} = x^{e^*,-}$ and $1<i<n$, we note that $x^{e^*}_i \in B_{e^*} \subset E_{e^*}$. Then $\dist (x^{e^*}_i, \partial E_{e^*}) \ge \dist \left(x^{e^*}_i, \partial B_{e^*} \right) = \sqrt{1-(e^*)^2} - |x^{e^*}_i| \ge \frac{3}{n}\sqrt{1-(e^*)^2} = 3\tau \left( x^{e^*} \right)$
    \item For $x^{e^*} = x^{e^*,+}$ and $1<i<n$ such that $|x^{e^*}_i|\ge (e^*)^2$, we have $\dist(x^{e^*}_i, \partial E_{e^*}) = 1-|x^{e^*}_i| \ge \frac{3}{n} = 3\tau \left( x^{e^*} \right)$ by Lemma \ref{lem: line}.
    \item For $x^{e^*} = x^{e^*,+}$ and $1 <i<n$ such that $|x^{e^*}_i|<(e^*)^2$, we have $\frac{|x^{e^*}_i|^2}{(e^*)^2} < |x^{e^*}_i|$. Then $\dist(x^{e^*}_i, \partial E_{e^*}) = \sqrt{\left( 1- \frac{|x^{e^*}_i|^2}{(e^*)^2} \right) \left(1-(e^*)^2 \right)} > \sqrt{\left( 1-|x^{e^*}_i| \right)\frac{1}{n}} > \sqrt{\frac{3}{n} \cdot \frac{1}{n}} = \frac{\sqrt{3}}{n} = \sqrt{3} \tau \left( x^{e^*} \right)$ by Lemma \ref{lem: line}.
\end{itemize}
For each of these distance functions $f$, by continuity of $g$, there exists $\delta_f>0$ such that for all $e\in [e^*, e^* + \delta_f)$, we have $g(e) > g(e^*) - \frac{1}{2} \ge \sqrt{3} - \frac{1}{2} > 1$. Therefore $N(x^e)$ does not contain the edge corresponding to $f$. Set $\delta = \min_f \delta_f$. Then $N(x^e) \cong N(x^{e^*})$ for all $e\in \left[ e^*, e^* + \delta \right)$, contradicting the definition of $e^*$.

Therefore $x^e = x^{e,+}, x^{e,-}$ are the unique families of configurations $\left(x^e \right)_{e}$ such that $x^e$ is a critical configuration of $E_e$ and $x^0 \in C$. 
\end{proof}

Provided the perturbation is small, we would expect the topology just above $r_+$ in the perturbed case to be the same as the topology just above $r_0$ in the unperturbed case. By contrast, the homotopy type on the region with radius just above $r_-$ may have different topology, and this is where the $\pi_{n-2}$-class of Baryshnikov et al. \cite[\S5.2]{bald} lives (if you consider an infinite strip as the limit of an ellipse with fixed semi-minor radius as eccentricity approaches 1). 

In this section, we investigate the topology beyond $r_+$, by considering the persistence of our homotopy class from \S\ref{sec: extend} under perturbations of $D^2$. First, we consider arbitrary small perturbations. Then, following Remark \ref{rem: thin}, we  consider to what extent it is necessary to use the full unit disc by investigating the persistence of this homotopy class as we replace $D^2$ with an ellipse of decreasing semi-minor radius. We will use the map $\ang$ from Def. \ref{defn: ang}.

\begin{thm} \label{thm: deform}
    Take $r\in (0,\infty)$, $n\ge 4$, and $U\subset \RR^2$. Suppose there are discs $C_1, C_2\subset\RR^2$ of radius $\left(n-1\right)r$ and strictly less than $nr$ respectively, such that $C_1\subset U$ and any disc of radius $r$ in $U$ is contained in $C_2$. Then there is some non-zero $\gamma \in \pi_{n-3}(\Cf_{n,r}(U))$ such that $\iota_\ast(\gamma)=0$ in $\pi_{n-3}(\Cf_n(U))$.
\end{thm}

In particular, this shows that the non-trivial $\pi_{n-3}$-class remains in $\Cf_{n,r}(U)$ for $\frac{1}{n} < r < \frac{1}{n-1}$ when $U$ is a deformation of $D^2$, provided that the deformation is small enough that: 1) we can fit a disc of radius $\left(n-1\right)r$ inside $U$, and 2) every disc of radius $r$ in $U$ lies in the unit disc. The first condition means the boundary of $U$ cannot be deformed very far towards the origin. On the other hand, these conditions place no restrictions on the area of $U$, nor do they require $U$ to be bounded -- consider, for example, the deformation of $D^2$ in which we remove the segment of $D^2$ with $y$-coordinate greater than $1-\varepsilon$ for some small $\varepsilon>0$, and then take the union of the remaining region with the semi-infinite strip $[-\varepsilon, \varepsilon] \times [0,\infty)$.

\begin{proof}
 The conditions on $C_1, C_2$ yield $\Cf_{n,r}(C_1) \subset \Cf_{n,r}(U) \subset \Cf_{n,r}(C_2)$. Denote the inclusion maps by $i_1, i_2$ respectively. Furthermore, it is impossible for $n$ discs of radius $r$ to lie on a straight line inside $C_2$, so $\ang(\Cf_{n,r}(C_2)) \subset T^{n-2}\backslash\{0\}$. 
 
 By Thm. \ref{thm: braid} and \ref{thm: nts}, there is some non-contractible sphere $\SS \subset \Cf_{n,r}(C_1)$ such that $\iota_\ast([\SS])=0$ in $\pi_{n-3}(\Cf_n(C_1))$. In particular, this is proven by showing that $\ang(\SS)$ is non-contractible in $T^{n-2}\backslash \{0\}$. We claim $\gamma \coloneqq (i_1)_\ast([\SS])$ satisfies the desired properties. 
 
 First, note that $\ang = \ang \circ i_2 \circ i_1$. Since $(\ang \circ i_2)_*\circ (i_1)_*([\SS]) = (\ang)_\ast[\SS] \neq 0$ in $\pi_{n-3}(T^{n-2} \backslash \{0\})$, it follows that $\gamma \neq 0$ in $\pi_{n-3}(\Cf_{n,r}(U))$.

 To finish, let $i_3 \colon \Cf_n(C_1) \hookrightarrow \Cf_n(U)$ be the inclusion map. Then $\iota \circ i_1 = i_3 \circ \iota$ and therefore \[\iota_\ast (\gamma) = \iota_\ast ((i_1)_\ast ([\SS])) = (i_3)_\ast (\iota_\ast ([\SS])) = (i_3)_\ast(0) = 0 \quad . \qedhere \]
\end{proof}

Following Remark \ref{rem: thin}, it is apparent we shouldn't need the full extent of the unit disc in Thm. \ref{thm: nts}. Indeed, if a small enough angle is used in place of $\frac{2\pi}{n}$ in the given construction of the $(n-3)$-sphere, the discs should remain within some narrow strip parallel to (and centred on) the $x$-axis. The following result shows that this homotopy class in \S3 remains when we deform the unit disc into an ellipse with semi-major radius 1, provided that the semi-minor radius is greater than $\frac{1}{\sqrt{n}}$. This result also enables us to generalise Thm. \ref{thm: deform} by replacing $C_1$ with a suitable ellipse. 

Under this deformation, the first critical level $\frac{1}{n}$ splits into two critical levels. The lower of these is $\frac{\sqrt{1-e^2}}{n}$, where the numerator is the semi-minor radius. This theorem concerns the upper critical level, $\frac{1}{n}$, which is related to the semi-major radius.

\begin{thm} \label{thm: ell}
    Let $n\in \NN$, and let $E$ be an ellipse of semi-major radius 1 and eccentricity $e\in \left[0,\sqrt{\frac{n-1}{n}}\right)$ (equivalently, the semi-minor radius of $E$ lies in $\left(\frac{1}{\sqrt{n}}, 1\right]$).  Then there is some $\varepsilon>0$ such that, for $r\in\left(\frac{1}{n}, \frac{1}{n}+\varepsilon\right]$, $\Cf_{n,r}(E)$ contains a non-contractible $(n-3)$-sphere.
\end{thm}

To prove this, we will construct a sphere in $\Cf_{n,r}(E)$, and prove that the disc centres in each configuration lie in some smaller ellipse $E'$. The following lemma defines the necessary relationship between the dimensions of $E$ and $E'$ which ensures that, if a disc has its centre in $E'$, then the disc is contained in $E$.

\begin{lem} \label{lem: small2}
    Take $\DD\in \Cf_{n,r}(\RR^2)$, $a>r$. Let $E'$ be an ellipse with semi-major radius $a-r$ and eccentricity $e'$ such that the centre of every disc of $\DD$ lies in $\overline{E'}$. Let $E$ be a concentric, coaxial ellipse of semi-major radius $a$ and eccentricity $e \le \sqrt{\left(1-\frac{r}{a}\right)\left(1-\sqrt{1-(e')^2}\right)}$. Then $\bigcup \DD \subset E$.
\end{lem}
\begin{proof}
    It is sufficient to check this for $e=\sqrt{\left(1-\frac{r}{a}\right)\left(1-\sqrt{1-(e')^2}\right)}$, since the associated ellipse is contained inside the ellipses of smaller eccentricity. For this proof, we use the equivalent definition $e=\sqrt{\frac{a-r-b'}{a}}$, where $b'=(a-r)\sqrt{1-(e')^2}$ is the semi-minor radius of $E'$. Since $E$ and $E'$ are concentric and coaxial, we choose coordinates so that the shared centre is the origin, and the shared major axis is the $x$-axis. Take $q=(q_1,q_2)\in \overline{E'}$ to be the centre of some disc in $\DD$ and assume by symmetry that $q$ lies in the first quadrant. We need to show that $\dist(q,\partial E)\ge r$. We will achieve this by finding a suitable point on the $x$-axis and using the triangle inequality.

    Let $q_3 = \min\{q_1,ae^2\}$ and $q'=(q_3,0)$. Since $q_3\le ae^2$, we have $\dist(q',\partial E) = \sqrt{\left(a^2-\frac{q_3^2}{e^2}\right) (1-e^2)}$  by Lemma \ref{lem: line}. This inequality also yields $a^2-\frac{q_3^2}{e^2} \ge a^2(1-e^2)$, so that $\dist(q', \partial E)^2 \ge a^2(1-e^2)^2 = a^2(1-\frac{a-r-b'}{a})^2 = (r+b')^2$. 
    
    Then it remains to show $|q-q'|\le b'$. If $q_1\le ae^2$, then $|q-q'|=q_2 \le b'$, where the second inequality arises from the fact that $b'$ is the semi-minor radius. Otherwise, $q' = (ae^2,0)$, and $\sup\{|q-q'| \colon q=(q_1,q_2) \in \overline{E'}, q_1>ae^2\}$ is achieved by some $(q_1,q_2) \in \partial E'$ with $q_1 \ge ae^2$. Since $ae^2\le (a-r)(e')^2$, Lemma \ref{lem: line} applied to $q'$ and $E'$ shows that $\partial {E'} \to \RR$, $q \mapsto |q-q'|$ has its local minima at $\left(\frac{p_1}{e^2}, \pm \sqrt{ \left( a^2 - \frac{p_1^2}{e^4} \right) \left( 1-e^2 \right)} \right)$, and therefore has local maxima at $(\pm(a-r),0)$. When we restrict to $q_1 \ge ae^2$, this supremum must be achieved at either: 1) $q_1=ae^2$, where $|q-q'|\le b'$ by the first case; or 2) $(a-r,0)$, where $|q-q'| = a-r-ae^2 = b'$.
    
    Thus $\dist(q,\partial E) \ge \dist(q',\partial E)-|q-q'| \ge \dist(q',\partial E)-b' \ge r$ as required.
\end{proof}

\begin{rem}
    The inequality $e \le \sqrt{\left(1-\frac{r}{a}\right)\left(1-\sqrt{1-(e')^2}\right)}$ is asymptotically sharp as $e'\to 1$. To see this, note that when $e'=1$, we have $e>\sqrt{1-\frac{r}{a}}$ if and only if $a-r<ae^2$. Hence, by Lemma \ref{lem: line}, $\dist((a-r,0), \partial D^2)$ is not achieved at $(a,0)$, so $\mathrm{dist}((a-r,0), \partial E) < |(a,0)-(a-r,0)| = r$, so the disc of radius $r$ centred at $(a-r,0) \in \overline{E'}$ is not contained in $E$. This is the case of interest in the following proof.
\end{rem}

\begin{proof}[Proof of Theorem \ref{thm: ell}]
    Take some angle $0<\xi<\frac{\pi}{n-2}$ such that $1-r>e^2$, where $r=\frac{1}{n-4\sin^2\frac{\xi}{4}}$, and let $S = \partial \left( \left[-\xi, \xi \right] \right)^{n-2} \subset T^{n-2}$. Let $e'=\sqrt{1-\left(1-\frac{e^2}{1-r}\right)^2}$, $z_1 = (-e'(1-r), 0)$, $z_2 = (e'(1-r),0)$. Note that $\frac{e^2}{1-r}< 1$ and $\frac{e^2}{1-r}\to 1$ as $e\to\sqrt{\frac{n-1}{n}}^-$, $\xi\to0$. Thus $e'<1$ and $e'\to 1$ as $e\to\sqrt{\frac{n-1}{n}}^-$, $\xi\to0$.
    We construct some $\SS \colon S\to \Cf_{n,r}(\RR^2)$ by the following algorithm on each $(\phi_2, \ldots, \phi_{n-1})\in S$:
    \begin{enumerate}
        \item Place $D_1$ at the origin of the plane and $D_2$ in contact with it, centred at $(2r, 0)$.
        \item Place each subsequent disc $D_i$ for $i\le n$ in contact with $D_{i-1}$ such that $x_i$ lies on the ray from $x_{i-1}$ at angle $\phi_{i-1}$ (Figure \ref{fig: disc i}).
        \item Translate and rotate the configuration so that $x_1$ and $x_n$ lie on the $x$-axis, equidistant from the origin.
    \end{enumerate}
    This is continuous and the discs are non-overlapping by the arguments in the proof of Theorem $\ref{thm: nts}$. 
    
    We now show that the centre of every disc in each configuration in $\SS$ fits inside the ellipse with semi-major radius $1-r$ and foci $z_1,z_2$, which has eccentricity $e'$. By remark \ref{rem: coord}, we see that, at stage 2, 
    \begin{align*}
        |x_1-x_n|^2 & = 4r^2 \left| \left( \sum_{i=1}^{n-1}\cos\theta_i \ , \ \sum_{i=1}^{n-1}\sin\theta_i \right) \right|^2 \\
        & = 4r^2 \left( \left( \sum_{i=1}^{n-1}\cos\theta_i\right)^2 + \left(\sum_{i=1}^{n-1}\sin\theta_i \right)^2 \right) \\
        & = 4r^2 \left( \sum_{i=1}^{n-1} (\cos^2\theta_i + \sin^2\theta_i) + 2\sum_{i=2}^{n-1}\sum_{j=1}^{i-1} (\cos\theta_i \cos\theta_j + \sin\theta_i \sin\theta_j) \right) \\
        & = 4r^2 \left( n-1 + 2\sum_{i=2}^{n-1}\sum_{j=1}^{i-1} \cos(\theta_i - \theta_j)  \right)
    \end{align*}
    where $|\theta_i-\theta_j| = |\sum_{k=i+1}^j \phi_k| \le (j-i)\xi \le (n-2)\xi$. Denoting $\theta = (n-2)\xi < \pi$, we see that $\cos(\theta_i-\theta_j) \ge \cos\theta$, and thus
    \begin{align*}
        |x_1-x_n|^2 & \ge 4r^2 \left( n-1 + 2\sum_{i=2}^{n-1}\sum_{j=1}^{i-1} \cos\theta \right) \\
        & = 4r^2 \left( n-1 + (n-1)(n-2) \cos\theta \right) \\
        & = 4r^2 \left( (n-1)^2 - (n-1)(n-2)(1-\cos\theta) \right) \quad .
    \end{align*}
    That is, at step 3, the $x$-coordinate of $x_n$ will be at least \[r\sqrt{(n-1)^2 - (n-1)(n-2)(1-\cos\theta)} \to \frac{n-1}{n} \textrm{ as } \xi \to 0 . \] Since 
    \begin{align*}
        e'(1-r)& = \sqrt{(1-r)^2-(1-r-e^2)^2} \\
        & = \sqrt{2(1-r)e^2-e^4} \\
        & < \sqrt{2\left(1-\frac{1}{n}\right)e^2-e^4} \\
        & = \sqrt{\left(1-\frac{1}{n}\right)^2-\left(1-\frac{1}{n}-e^2\right)^2} \\
        & < 1-\frac{1}{n} \\
        & = \frac{n-1}{n}
    \end{align*} and $\sqrt{2\left(1-\frac{1}{n}\right)e^2-e^4}$ is independent of $\xi$, it follows that \[r\sqrt{(n-1)^2 - (n-1)(n-2)(1-\cos\theta)} > e'(1-r)\] for sufficiently small $\xi$. Then $x_1, x_n$ are further from the origin than $z_1, z_2$ for small $\xi$, so we can find $t\in(0,1)$ such that $z_1 = tx_1 + (1-t)x_n$ and $z_2 = (1-t)x_1 + tx_n$.

    Using this, we find that for any $i$, 
    \begin{align*}
        |x_i-z_1| + |x_i-z_2| & \le t|x_i-x_1| + (1-t)|x_i-x_n| + (1-t)|x_i-x_1| + t|x_i-x_n| \\
        & = |x_i-x_1| + |x_i-x_n| \\
        & \le 2r\left(n-1-4\sin^2\frac{\xi}{4}\right) \quad ,
    \end{align*}
    where the last line follows from Lemma \ref{lem: fit}. That is, each $x_i$ lies inside the closed ellipse with foci $z_1,z_2$ and semi-major radius $(n-4\sin^2 \frac{\xi}{4})r - r = 1-r$ as required.
    
    Therefore, since $e = \sqrt{(1-r)\left(1-\sqrt{1-(e')^2}\right)}$, we observe that for all $\Phi \in S$, $\bigcup \SS(\Phi) \subset E$ by Lemma \ref{lem: small2}. Thus $\SS$ is a local section of $\ang \colon \Cf_{n,r}(E) \to T^{n-2}$.

    We finish by observing that $E \subset D^2$, and moreover $[S] \neq 0 \in \pi_{n-3}(T^{n-2}\backslash \{0\})$, where $0\in T^{n-2}$ corresponds to all discs lying in a straight line. Therefore $\SS$ represents a non-trivial class in $\Cf_{n,r}(E)$ by Lemma \ref{lem: punc}.
\end{proof}

While $\frac{1}{\sqrt{n}}$ is the best semi-minor radius we can achieve while keeping the semi-major radius equal to 1, we would hope to achieve better, since, as made clear in Remark \ref{rem: thin}, the maximal distance of our configurations from the $y$-axis is arbitrarily close to $\frac{1}{n}$. We achieve this semi-minor radius in the next theorem. However, in order to achieve this, we must allow our semi-major radius, and thus too the area of the ellipse, to approach infinity.

\begin{thm} \label{thm: ell2}
    Let $n\in \NN$, and let $E$ be an ellipse with eccentricity $e\in \left[\sqrt{1-r},1\right)$ and semi-minor radius $b=\sqrt{ \frac{(1-r)^2(1-e^2)}{e^2} + r^2}$, where $r=\frac{1}{n}+\varepsilon$ for some sufficiently small $\varepsilon>0$. Then $\Cf_{n,r}(E)$ contains a non-contractible $(n-3)$-sphere.
\end{thm}
Under these hypotheses, $b$ takes all values in $\left(r, \sqrt{r}\right]$. The semi-major radius is $b\left(1-e^2\right)^{-\frac{1}{2}} = \sqrt{ \frac{(1-r)^2}{e^2} + \frac{r^2}{1-e^2} }$, which takes all values in $\left[1, \infty\right)$. 
\begin{proof}
    For a sufficiently small angle $\xi>0$, let $S=\partial \left[-\xi, \xi\right]^{n-2} \subset T^{n-2}$ and take $\frac{1}{n}<r \le \frac{1}{n-4\sin^2\frac{\xi}{4}}$. Take the length-preserving coordinate system in which $E$ is centred at the origin and its major axis coincides with the $x$-axis. We define a local section $\SS\colon S\to \Cf_{n,r}(E)$ to $\ang$ by constructing each $\SS(\phi_2, \ldots, \phi_{n-1})$ as follows: 
    \begin{enumerate}
    \item Place $D_1$ at the origin of the plane and $D_2$ in contact with it at $(2r, 0)$.
    \item Place each subsequent disc $D_i$ in contact with $D_{i-1}$, such that $x_i$ lies on the ray from $x_{i-1}$ at angle $\phi_{i-1}$.
    \item Translate all discs by $-\frac{x_1+x_n}{2}$.
    \item Rotate the configuration about the origin so that $x_1$ and $x_n$ lie on the $x$-axis.
\end{enumerate}

There is no overlap between discs and $\SS$ is continuous by the arguments of Thm. \ref{thm: nts}, so we simply need to show $\bigcup \SS(\Phi) \subset E$ for all $\Phi \in S$. By Remark \ref{rem: thin}, $x_i$ is arbitrarily close to $(\frac{2i-n-1}{n}, 0)$ at step 3. In particular, $x_1 \to (-\frac{n-1}{n}, 0)$ and $x_n\to (\frac{n-1}{n},0)$ as $\xi \to 0$, so the size of the rotation in step 4 tends to 0, which means this limit position holds after step 4 as well. This limit is the first critical configuration, consisting of the $n$ discs of radius $\frac{1}{n}$ lined up along the horizontal diameter of $D^2$, such that $D_i$ touches $D_{i-1}$ and $D_{i+1}$ if they exist, and $D_1$ and $D_n$ touch the boundary.

More precisely, for each $\frac{1}{n}<r \le \frac{1}{n-4\sin^2\frac{\xi}{4}}$, there is some $s>r$ (where $s\to \frac{1}{n}$ as $\xi\to 0$) such that, for $2\le i \le n-1$, $D_i \subset B\left(\left(\frac{2i-n-1}{n},0\right), s\right)$. Therefore \[\bigcup_{i=2}^{n-1} D_i \subset F_1 \coloneqq \left[-\frac{n-3}{n}-s, \frac{n-3}{n}+s\right] \times [-s,s] \quad .\] 
We have fixed the centres of $D_1$ and $D_n$ to lie on the $x$-axis, and Lemma \ref{lem: fit} gives us $|x_1|, |x_n| \le (n-1-4\sin^2\frac{\xi}{4})r \le 1-r$. Thus \[ D_1\cup D_n \subset F_2 \coloneqq B((-1+r,0), r) \cup ([-1+r,1-r] \times [-r,r]) \cup B((1-r,0),r) \quad . \]
For sufficiently small $\xi$, we have $s<\frac{3}{2n}$. Then $\frac{n-3}{n}+s < \frac{2n-3}{2n} < 1-r$, and so \[ \bigcup \SS(\Phi) \subset F_1\cup F_2 \subset F\coloneqq B((-1+r,0), r) \cup ([-1+r,1-r] \times [-s,s]) \cup B((1-r,0),r)\] for all $\Phi \in S$ (see Fig. \ref{fig: thicken}).

\begin{figure}[htbp]
    \begin{center}
        \begin{tikzpicture}
            \foreach \x in {-5,-3,...,5} 
                {\draw ({0.8*\x},0) circle[radius=1];}
            \foreach \x in {-5.5,5.5} 
                {\draw[thick] ({\x},0) circle[radius=0.9];}
            \draw[thick] (-5.5,-1) rectangle (5.5,1);
        \end{tikzpicture}
    \end{center}
        \caption{The outer boundary shows a thickening of the first critical configuration, constructed in such a way that it contains the configurations of the non-contractible sphere $\SS$, and is in turn contained in the ellipse $E$, in Theorem \ref{thm: ell2}.}
        \label{fig: thicken}
    \end{figure}

Next, we show that $F\subset E$ by showing that all points of $F$ satisfy the inequality which defines $E$, \[x^2(1-e^2)+y^2\le b^2 = \frac{(1-r)^2(1-e^2)}{e^2} + r^2 \quad .\] 

First, if $(x,y) \in [-1+r,1-r] \times [-s,s]$, then $x^2(1-e^2)+y^2 \le (1-r)^2(1-e^2)+s^2$. As $\xi \to 0$, then $s,r\to \frac{1}{n}$, and therefore $s^2-r^2 \to 0$. Since the lower bound on $e$ is always less than $\sqrt{\frac{n-1}{n}}$, we may assume $e$ to be fixed. Therefore $s^2-r^2 \le \frac{1}{4}(1-e^2)\left(\frac{1}{e^2}-1\right) \le (1-r)^2(1-e^2)\left(\frac{1}{e^2}-1\right)$ for sufficiently small $\xi$. Hence 
\begin{align*}
    x^2(1-e^2)+y^2 & \le (1-r)^2(1-e^2) + r^2 + (1-r)^2(1-e^2)\left(\frac{1}{e^2}-1\right) \\
    & = \frac{(1-r)^2(1-e^2)}{e^2} + r^2 \quad .
\end{align*}

Second, if $(x,y) \in B((1-r,0),r)$, then $(x-1+r)^2+y^2 \le r^2$. Thus 
\begin{align*}
    x^2(1-e^2)+y^2 & \le x^2(1-e^2) + r^2 - (x-1+r)^2 \\
    & = -e^2x^2+2(1-r)x+2r-1 \\
    & = -\left(ex-\frac{1-r}{e}\right)^2 + \frac{(1-r)^2}{e^2} + 2r-1 \\
    & \le \frac{(1-r)^2}{e^2} + 2r-1 \\
    & = \frac{(1-r)^2(1-e^2)}{e^2} + (1-r)^2 + 2r-1 \\
    & = \frac{(1-r)^2(1-e^2)}{e^2} + r^2 \quad .
\end{align*}
Symmetrically, we see the same for $(x,y) \in B((-1+r,0),r)$. Therefore $\bigcup \SS(\Phi) \subset F \subset E$ for all $\Phi \in S$, which completes the proof that $\SS$ is a local section of $\ang$.

Finally, let $\Delta = \{B(p,r) \colon p\in \RR^2, \ B(p,r) \subset E\}$ be the set of all discs of radius $r$ contained in $E$. We want to show that every such disc also lies in $D^2$. It is clear geometrically that $\sup \{|p| \colon B(p,r) \in\Delta \}$ is achieved on the major axis of $E$, which is the $x$-axis. We may assume the $x$-coordinate of $p$ to be positive by symmetry. Then, by checking that these coordinates satisfy the equations defining both sets, we note that $\partial B((1-r,0),r)$ intersects $\partial E$ at $p_{\pm}=\left(\frac{1-r}{e^2}, \pm \sqrt{r^2-(1-r)^2(e^{-2}-1)^2}\right)$, so the $x$-coordinate of $p$ cannot be greater than $1-r$ -- that is, $\sup \{|p| \colon B(p,r) \in\Delta \} \le 1-r$. This is equivalent to the claim. 

Now $[S]\neq 0 \in \pi_{n-3}(T^{n-2}\backslash \{0\})$, and if $\ang(\DD) = 0$, then all the disc centres are collinear. Thus $\SS$ represents a non-trivial element of $\pi_{n-3}(\Cf_{n,r}(E))$ by Lemma \ref{lem: punc}, as required.
\end{proof}

We note that there is no discontinuity in the ranges of the semi-minor and semi-major radii between Theorems \ref{thm: ell} and \ref{thm: ell2}; indeed, when $e=\sqrt{1-r}$, both settings yield the ellipse of semi-major radius 1 and semi-minor radius $\sqrt{r}$. Hence there exists a continuous family of ellipses, indexed over the semi-minor radius on the range $\left(\frac{1}{n}, 1\right]$, which contain this non-contractible sphere class.

\bibliography{library}

@article{cohom,
    author =       "Arnol'd, Vladimir Igorevich",
    title =        "The cohomology ring of the coloured braid group",
    journal =      "Mathematical Notes of the Academy of Sciences of the USSR",
    volume =       "5",
    number =       "2",
    pages =        "138--140",
    year =         "1969",
    DOI =          "https://doi.org/10.1007/BF01098313"
}

@incollection{bott,
    author =       "Guest, Martin A",
    editor="Bridson, Martin R and Salamon, Simon M",
    title="Morse Theory in the 1990's",
    booktitle="Invitations to Geometry and Toplogy",
    year="2002",
    publisher="Oxford University Press",
    address="Oxford",
    pages="146--207",
    isbn="9780198507727",
}

@article{bald,
    author =       "Baryshnikov, Yuliy and Bubenik, Peter and Kahle, Matthew",
    title =        "Min-Type Morse Theory for Configuration Spaces of Hard Spheres",
    journal =      "International Mathematical Research Notices",
    volume =       "2014",
    number =       "9",
    pages =        "2577--2592",
    year =         "2014",
    DOI =          "https://doi.org/10.1093/imrn/rnt012"
}

@article{alp,
    author =       "Alpert, Hannah",
    title =        "Restricting Cohomology Classes to Disk and Segment Configuration Spaces",
    journal =      "Topology and its Applications",
    volume =       "230",
    number =       "2",
    pages =        "51--76",
    year =         "2016",
    DOI =          "https://doi.org/10.1016/j.topol.2017.08.004"
}

@article{aspher,
    author =       "Fadell, Edward and Neuwirth, Lee",
    title =        "Configuration Spaces",
    journal =      "Mathematica Scandinavica",
    volume =       "10",
    pages =        "111--118",
    year =         "1962"
}

@misc{ell,
author =           "Eberly, David",
title =       "Distance from a Point to an Ellipse, an Ellipsoid, or a Hyperellipsoid",
year =             "2013",
howpublished = "\url{https://www.geometrictools.com/Documentation/Documentation.html}",
note =    "Accessed: 4th September 2023",
}

@article{sq3,
Author = {Alpert, Hannah and Kahle, Matthew and Macpherson, Robert},
Title = {Asymptotic Betti Numbers for Hard Squares in the Homological Liquid
   Regime},
Journal = {International Mathematics Research Notices},
Year = {2024},
volume = "2024",
issue = "10",
pages = "8240--8263"
}

@article{comptop,
  author               = {Carlsson, Gunnar and Gorham, Jackson and Kahle, Matthew and Mason, Jeremy},
  journal              = {Physical Review E},
  title                = {Computational topology for configuration spaces of hard disks},
  year                 = {2012},
  pages                = {011303},
  issn                 = {1539-3755},
  number               = {1},
  volume               = {85},
  article-number       = {011303},
  doi                  = {10.1103/PhysRevE.85.011303},
}

@article{disk1,
  author               = {Alpert, Hannah and Kahle, Matthew and MacPherson, Robert},
  journal              = {Journal of Applied and Computational Topology},
  title                = {Configuration spaces of disks in an infinite strip},
  year                 = {2021},
  volume               = {5},
  number               = {3},
  pages                = {357--390},
}

@article{disk2,
    author =       "Alpert, Hannah",
    title =        "Generalized representation stability for disks in a strip and no-k-equal spaces",
    journal =      "arXiv:2006.01240",
    year =         "2020",
}

@article{disk3,
    author =       "Alpert, Hannah and Manin, Fedor",
    title =        "Configuration spaces of disks in a strip, twisted algebras, persistence, and other stories",
    journal =      "Geometry and Topology",
    volume =       "28",
    number =       "2",
    pages =        "641--699",
    year =         "2021",
}

@article{disk4,
    author =       "Wawrykow, Nicholas",
    title =        "On the symmetric group action on rigid disks on a strip",
    journal =      "Journal of Applied and Computational Topology",
    volume =       "7",
    number =       "3",
    pages =        "427--472",
    year =         "2022",
}

@article{disk5,
    author =       "Wawrykow, Nicholas",
    title =        "Representation Stability for Disks in a Strip",
    journal =      "Journal of Topology and Analysis",
    volume =       "0",
    number =       "0",
    year =         "2024",
    doi = "10.1142/S1793525324500250",
}

@article{disk6,
    author =       "Wawrykow, Nicholas",
    title =        "The topological complexity of the ordered configuration space of disks in a strip",
    journal =      "Proceedings of the American Mathematical Society Series B",
    volume =        "11",
    pages =         "638--652",
    year =         "2024",
}

@article{sq1,
    author =       "Alpert, Hannah",
    title =        "Discrete configuration spaces of squares and hexagons",
    journal =      "Journal of Applied and Computational Topology",
    volume =       "4",
    number =       "2",
    pages =        "263--280",
    year =         "2019",
}

@article{sq2,
    author =       "Alpert, Hannah and Bauer, Ulrich and Kahle, Matthew and MacPherson, Robert and Spendlove, Kelly",
    title =        "Homology of configuration spaces of hard squares in a rectangle",
    journal =      "Algebraic and Geometric Topology",
    volume =       "23",
    number =       "6",
    pages =        "2593--2626",
    year =         "2023",
}

@article{sq4,
    author =       "Plachta, Leonid",
    title =        "Configuration spaces of squares in a rectangle",
    journal =      "Algebraic and Geometric Topology",
    volume =       "21",
    number =       "3",
    pages =        "1445--1478",
    year =         "2021",
}

@article{bors,
    author =       "Borsuk, Karol",
    title =        "Sur les prolongements des transformations continues",
    journal =      "Fundamnta Mathmaticae",
    volume =       "28",
    number =       "1",
    pages =        "99--110",
    year =         "1937",
}

@article{hann,
    author =       "Hanner, Olof",
    title =        "Some theorems on absolute neighbourhood retracts",
    journal =      "Arkiv f{\"o}r Matematik",
    volume =       "1",
    number =        "30",
    pages =        "389--408",
    year =         "1951",
}

@article{loj,
    author =       "Lojasiewicz, S",
    title =        "Triangulation of semi-analytic sets",
    journal =      "Annali della Scuola Normale Superiore di Pisa, Classe di Scienze 3 e série",
    volume =       "18",
    number =       "4",
    pages =        "449--474",
    year =         "1964",
}

@book{hatch,
author =           "Hatcher, Allen",
title =            "Algebraic Topology",
year =             "2001",
city =             "Cambridge",
isbn =              "978-0-521-79160-1",
publisher =        "Cambridge University Press",
}

@inbook{brown,
author =           "Brown, Ronald",
title =            "Elements of Modern Topology",
year =             "1968",
chapter =           "2",
pages =             "21--30",
city =             "London",
isbn =              "978-0070940598",
publisher =        "McGraw-Hill",
}

@inbook{hans,
author =           "Hansen, Vagn Lundsgaard",
title =            "Braids and Coverings",
year =             "1989",
chapter =           "I",
pages =             "2--49",
city =             "Cambridge",
isbn =              "978-0070940598",
publisher =        "Cambridge University Press",
}

@article{top1,
    author =       "Teixeira, A C R and Stariolo, D A",
    title =        "Topological hypothesis on phase transitions: The simplest case",
    journal =      "Physical Review E",
    volume =       "70",
    number =       "1",
    pages =        "016113",
    year =         "2004",
}

@article{top2,
    author =       "Angelani, L and Casette, L and Pettini, M and Ruocco, G and Zamponi, F",
    title =        "Topology and phase transitions: From an exactly solvable model to a relation between topology and thermodynamic",
    journal =      "Physical Review E",
    volume =       "71",
    number =       "3",
    pages =        "036152",
    year =         "2005",
}

@inbook{farb,
author =           "Farber, Michael",
title =            "Invitation to Topological Robotics",
year =             "2004",
chapter =           "4",
pages =             "87--123",
city =             "Z{\"u}rich",
isbn =              "978-3037190548",
publisher =        "European Mathematical Society",
}

\end{document}